\documentclass[a4paper, 11pt, twoside]{article}
\usepackage[a4paper,margin=2.5cm]{geometry}
\usepackage[utf8]{inputenc}
\usepackage{lmodern}
\usepackage{amssymb}
\usepackage[english]{babel}
\usepackage{graphicx} 
\usepackage{epsfig}
\usepackage{epstopdf}
\usepackage{amsfonts}
\usepackage{amsmath}  
\usepackage{braket}
\usepackage{amsthm} 
\usepackage{chngpage}
\usepackage[usenames,dvipsnames]{xcolor}
\usepackage{mathrsfs}
\usepackage{enumitem}
\usepackage{comment}
\usepackage{chngcntr}
\usepackage[final]{microtype}

\usepackage{todonotes}
\usepackage{array}
\usepackage{caption}
\usepackage{subcaption}
\usepackage[T1]{fontenc}
\usepackage{listings}
\usepackage{tikz}
\usepackage{tikz-cd}
\usepackage{hyperref}
\hypersetup{
    colorlinks,
    linkcolor={red!50!black},
    citecolor={blue!50!black},
    urlcolor={blue!80!black}
}
\usepackage{cleveref}

\lstdefinelanguage{Julia}
  {morekeywords={abstract,break,case,catch,const,continue,do,else,elseif, 
      end,export,false,for,function,immutable,import,importall,if,in, 
      macro,module,otherwise,quote,return,switch,true,try,type,typealias, 
      using,while}, 
   sensitive=true, 
   morecomment=[l]\#, 
   morecomment=[n]{\#=}{=\#}, 
   morestring=[s]{"}{"}, 
   morestring=[m]{'}{'}, 
}[keywords,comments,strings] 

\newcommand{\C}{\mathbb{C}}

\newcommand{\R}{\mathbb{R}}

\newcommand{\ME}{{\mathcal E}}
\newcommand{\MS}{{\mathcal S}}
\newcommand{\MX}{{\mathcal X}}
\newcommand{\MY}{{\mathcal Y}}
\newcommand{\MF}{{\mathcal F}}

\newcommand{\MET}{\text{\rm MET}}
\newcommand{\MK}{{\mathcal K}}
\newcommand{\MC}{{\mathcal C}}
\newcommand{\MCscr}{{\mathscr {C}}}
\newcommand{\MH}{{\mathcal H}}
 \newcommand{\Row}{{\text{\rm Row}}}
\newcommand{\ML}{{\mathcal L}}
\newcommand{\Z}{\mathbb Z}
\newcommand{\N}{\mathbb N}
\newcommand{\one}{{\mathbf 1}}

\newcommand{\ao}{\mathbf {1}}
\newcommand{\azero}{\mathbf{0}}

\newcommand{\tME}{\widetilde \ME}

\theoremstyle{plain}
	\newtheorem{theorem}{Theorem}[section] 	
	\newtheorem{lemma}[theorem]{Lemma}
	\newtheorem{corollary}[theorem]{Corollary}
	\newtheorem{proposition}[theorem]{Proposition}
    
     \newtheorem{definition}[theorem]{Definition}
     \newtheorem{example}[theorem]{Example}
	\newtheorem{remark}[theorem]{Remark}
    
\setlist{itemsep=2pt, topsep=3pt, parsep=0pt}

\title{\bf The Algebraic Boundary of Graph Elliptopes}
\author{Monique Laurent \thanks{Centrum Wiskunde \& Informatica (CWI), Amsterdam, and Tilburg University, \url{Monique.Laurent@cwi.nl}}
\and Francesco Maria Mascarin \thanks{MPI-MiS Leipzig, and CWI, Amsterdam,
\url{francesco.mascarin@mis.mpg.de}}
\and  Simon Telen \thanks{MPI-MiS Leipzig,
\url{simon.telen@mis.mpg.de}}
}
\date{}
 
\begin{document}
 \maketitle

\begin{abstract}
\noindent This paper studies the algebraic boundary of the elliptope $\ME(G)$ of a graph $G$. In particular, we completely characterize the algebraic boundary of $\ME(G)$ when $G$ is cycle completable. In this case, the boundary is a union of determinantal hypersurfaces and Lissajous varieties, i.e., images of rational linear subspaces under the coordinatewise cosine map. 
As an application, we show that the algebraic boundary of  $\ME(G)$ is disjoint from its interior precisely when $\ME(G)$ is a spectrahedron or, equivalently, when $G$ is a chordal graph.
A central ingredient for the defining equation of the boundary hypersurface is the cycle polynomial, which captures the algebraic boundary of the elliptope $\ME(C_n)$ of the 
 $n$-th cycle graph $C_n$. We show that the cycle polynomial of $C_n$ is the resultant of two smaller cycle polynomials. Via this result, Sylvester's determinantal formula offers an inductive method for computing the cycle polynomial which mirrors a geometric property of metric polytopes.
 We also determine the degree of the homogeneous cycle polynomial, settling an open question of Sturmfels and Uhler~(2010). 
\end{abstract}

\section{Introduction}\label{sec:introduction}

Given a (simple undirected) graph $G=(V=[n],E)$, its  {\em elliptope} $\ME(G)$ is defined as the projection onto $\R^E$ of the set $\ME_n$ of correlation matrices. In symbols, we have
\begin{align*}
\ME_n & \, = \, \{X\in \MS^n_+:     X_{ii}=1 \text{ for } i\in [n]\},\\
 \ME(G) & \,  = \,  \{x\in \R^E: \exists X \in \ME_n \text{ such that } X_{ij}=x_{ij} \text{ for } \{i,j\}\in E\}.
\end{align*}
Here, $\MS^n_+$ is the cone of $n\times n$ real symmetric matrices $X\in\MS^n$ that are  positive semidefinite, denoted as $X\succeq 0$. The elliptope $\ME(K_n)\simeq \ME_n$ of the  complete graph $G=K_n$ is a spectrahedron. It is the convex semialgebraic set consisting of the matrices $X \in \MS^n$ with ones on the diagonal and with non-negative principal minors. By the Tarksi-Seidenberg theorem, any graph elliptope $\ME(G)$ is semialgebraic as well, and $\ME(G)$ is a spectrahedral shadow by definition. Moreover, $\ME(G) \subseteq \R^E$ is convex and it has non-empty interior: ${\rm int}(\ME(G)) \neq \emptyset$. In this paper we investigate graph elliptopes from an algebraic geometry perspective. More precisely, we study their algebraic boundary $\partial_a\ME(G)$, which is the Zariski closure of the euclidean boundary $\partial\ME(G) = \ME(G) \setminus {\rm int}(\ME(G))$ in $\mathbb{C}^E$. For instance, the algebraic boundary of $\ME_n$ is given by the polynomial equation $\det(X)=0$. 

The paper  \cite{Sturmfels2010MultivariateGaussians} discusses the algebraic boundary of elliptopes of chordal graphs and cycles. It is stated in \cite[Proposition 4.1]{Sturmfels2010MultivariateGaussians} that, if $G$ is chordal, then $\partial_a \ME(G)$ is defined by the determinantal equations corresponding to the maximal cliques of $G$. In \cite[Section 4.2]{Sturmfels2010MultivariateGaussians} it is stated (without proof) that, if $G$ is a cycle $C_n$, then $\partial_a \ME(C_n)$ is defined by the so-called $n$-th cycle polynomial $\Gamma_n$.
One of our main results (Theorem \ref{thm:boundary-cycle-completable}) is that determinantal equations and cycle polynomials suffice to describe the algebraic boundary $\partial_a \ME(G)$ for any cycle completable graph $G$. This includes chordal graphs and graphs with no $K_4$-minor, as well as any clique sum of such graphs (see Theorem \ref{theo:cycle-completable}). As a byproduct, we show that the elliptope $\ME(G)$ is a spectrahedron if and only if its algebraic boundary does not intersect its interior, which holds precisely when $G$ is chordal (Corollary \ref{cor:spectrahedron}). In particular, the elliptope of a cycle $C_n$ with length $n\ge 4$ is not a spectrahedron, which rectifies a claim made in \cite[Section~4.2]{Sturmfels2010MultivariateGaussians}. 

The affine hypersurface defined by $\Gamma_n = 0$ ($n\ge 3$) is the image of a rational hyperplane under the coordinatewise cosine map. Here ``rational'' means defined over $\mathbb{Q}$. Such varieties were recently studied in  \cite{BelAfiaMeroniTelen2025Chebyshev,mascarin2025lissajousvarieties}. In particular, \cite{mascarin2025lissajousvarieties} introduces the name ``Lissajous variety'' for the cosine of a rational affine space $L$. If $L$ is the row space of the incidence matrix of the cycle $C_n$, then one obtains the Lissajous variety $\ML_n = \cos(L) =  V(\Gamma_n)$. 
So, Lissajous hypersurfaces defined by cycle polynomials play a central role in describing algebraic boundaries of graph elliptopes.
This connection is readily seen from the well-known result from \cite{Barrett1993realpositivedefinite}: the elliptope of a cycle is the image of a convex polytope, namely the metric polytope, under the cosine map.

Cycle polynomials have remarkable properties. They are symmetric and, as shown in \cite{mascarin2025lissajousvarieties}, their associated Lissajous hypersurface $\ML_n$  admits a determinantal representation using multiplication matrices in a suitable polynomial algebra.
 For $n=3$, $\ME(C_3)\simeq\ME_3$ and  $\Gamma_3$ recovers the  polynomial $\det(X)$ for $X\in\ME_3$. For $n\ge 4$, we show that  the cycle polynomial $\Gamma_n$ can be obtained by computing  the resultant of two smaller cycle polynomials $\Gamma_p$ and $\Gamma_q$ where $n=p+q-2$ and $p,q\ge 3$. In a nutshell, one should think of adding a chord to $C_n$, which creates two smaller cycles sharing the chord as a common edge. The cycle polynomial of $C_n$ is obtained by eliminating the variable associated to the new chord from the two cycle polynomials of the smaller cycles. 
 
\begin{example} \label{ex:intro}
The elliptope $\ME_3 = \ME(C_3) = \ME(K_3) \subseteq \mathbb{R}^3$ is the semialgebraic set 
\[ \Big \{ (x_1,x_2,y) \in \mathbb{R}^3 \, : \, X\, = \,  \left ( \begin{smallmatrix} 1 & x_1 & x_2 \\ x_1 & 1 & y \\ x_2 & y & 1 \end{smallmatrix} \right ) \succeq 0 \Big \}. \]
Its algebraic boundary is Cayley's cubic surface, given by $f(x_1,x_2,y) = \det (X) = 0$, with $f(x_1,x_2,y)=-\Gamma_3(x_1,x_2,y)$ (see Example \ref{ex:smallGamma}). 
This surface is shown at the bottom-right of Figure \ref{fig:E4} (rightmost picture in the third row).
The elliptope $\ME(C_4)$ of the cycle $C_4$ consists of all points $(x_1,x_2,x_3,x_4) \in \mathbb{R}^4$ for which the partially filled matrix 
\[ \left ( \begin{smallmatrix}
1 & x_1 & ? & x_2 \\ 
x_1 & 1 & x_3 & ? \\ 
? & x_3 & 1 & x_4 \\ 
x_2 & ? & x_4 & 1
\end{smallmatrix} \right ) \]
admits a positive semidefinite completion. Its algebraic boundary is given by the polynomial $(1-x_1^2)(1-x_2^2)(1-x_3^2)(1-x_4^2) \cdot \Gamma_4$, where the fourth cycle polynomial $\Gamma_4$ is the resultant 
\[ \Gamma_4 \, = \, {\rm Res}_y(f(x_1,x_2,y),f(x_3,x_4,y)). \]
The expansion of $\Gamma_4$ is displayed in Example \ref{ex:smallGamma}. The four factors $(1-x_i^2)$ come from the maximal cliques of $C_4$, i.e., from the four edges. 
The point $x = {\bf 0}$ is an interior point of $\ME(C_4)$, and a zero of $\Gamma_4$,  i.e.,  $\Gamma_4({\bf 0}) = 0$. The algebraic boundary of a spectrahedron does not intersect the interior, which shows that $\ME(C_4)$ is not a spectrahedron (see Corollary \ref{cor:spectrahedron}). The second row of Figure~\ref{fig:E4} shows three different slices in $\R^3$ of the elliptope $\ME(C_4)\subseteq \R^4$, obtained by setting $x_4=0,\tfrac{1}{2},1$, respectively, and its third row shows   the corresponding slices of the real variety 
$V_\R(\Gamma_4)=V(\Gamma_4)\cap\R^4$. The first row of Figure~\ref{fig:E4} shows the corresponding slices of the metric polytope $\MET(C_4)=\tfrac{1}{\pi}\arccos(\ME(C_4))$, obtained by setting $\tfrac{1}{\pi}\arccos x_4=\tfrac{1}{2},\tfrac{1}{3},0$. See Section~\ref{sec:prel-elliptope} for details on the metric polytope.
\end{example}

\begin{figure}
\centering
\includegraphics[height = 5cm]{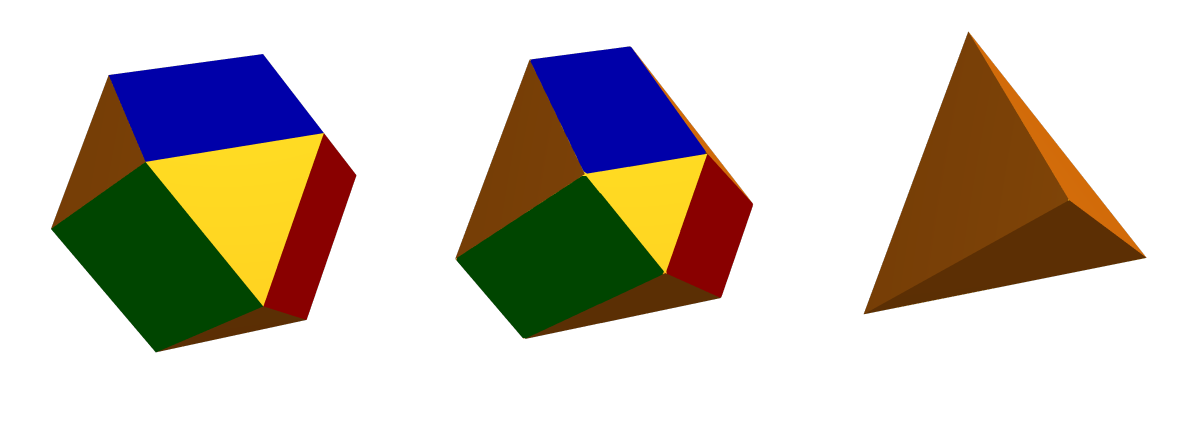} 
\includegraphics[height = 5cm]{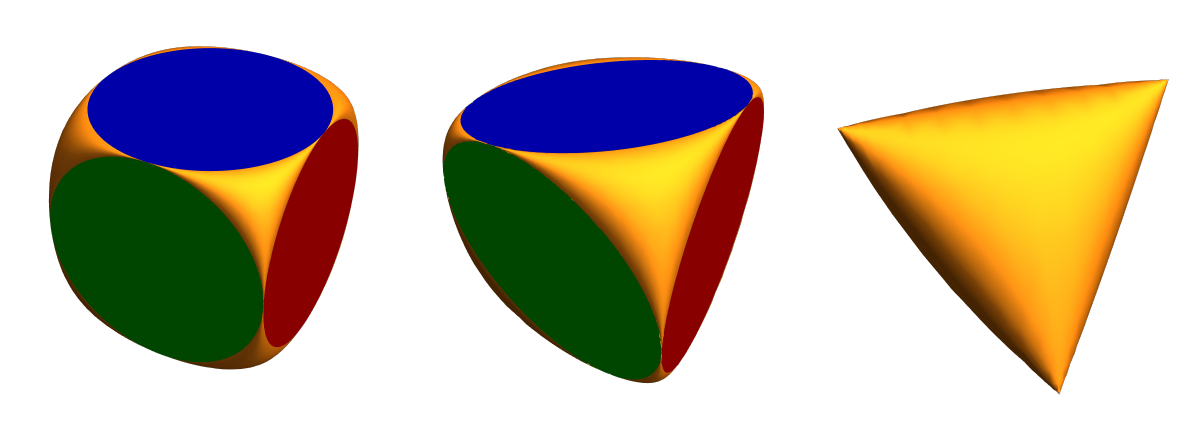} 
\includegraphics[height = 5cm]{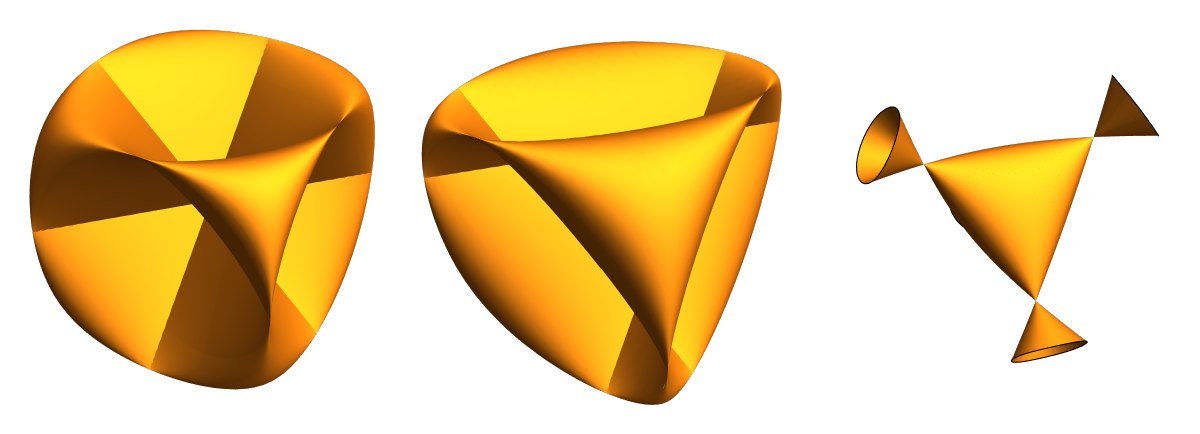} 
\caption{Three slices of $\ME(C_4)$ (second row) and the corresponding slices of $V_\R(\Gamma_4)$ (third row) 
at $x_4 = 0,\tfrac{1}{2},1$,
together with the matching slices of $\MET(C_4)$ at ${1\over \pi}\arccos x_4 = \tfrac{1}{2},\tfrac{1}{3},0$ (first row).}
\label{fig:E4}
\end{figure}

In \cite{mascarin2025lissajousvarieties} the authors compute the degree of $\Gamma_n$. Here, we prove that the degree of its homogeneous analog $\Gamma_n^h$ is $n \cdot 2^{n-3}$ (Theorem \ref{thm:degGammah}), answering an open question posed in \cite[Conjecture~4.9]{Sturmfels2010MultivariateGaussians}. The polynomial $\Gamma_n^h$ arises as follows: for the cycle $C_n=(V=[n],E_n)$, 
the cone $\MCscr(C_n)=\pi_{V \cup E_n}(\MS^n_+) \subseteq \R^{V\cup E_n}$ is the image of the positive semidefinite cone under the projection $\pi_{V \cup E_n}(X) = (\pi_V(X), \pi_{E_n}(X)) \in \R^{V \cup E_n}$, where $\pi_V(X) = (X_{11}, \ldots, X_{nn})$ is the diagonal of $X$. The algebraic boundary of $\MCscr(C_n)$ involves the polynomial $\Gamma_n^h$. Our strategy for computing the degree of $\Gamma_n^h$ implies a stronger result, which states the multidegree of $\Gamma_n^h$ with respect to a natural $\mathbb{Z}^n$-grading (Theorem \ref{thm:multidegGammah}). 

We now give an overview of relevant background literature and applications. 

\paragraph{Some background literature and applications.} Graph elliptopes $\ME(G)$ arise in several different contexts. First, as we saw in Example \ref{ex:intro}, they are directly connected to the {positive semidefinite (PSD) completion problem}. Given a partially filled correlation matrix whose   $(i,j)$-th entries  are known for $\{i,j\} \in E$, the PSD completion problem asks whether it can be completed to a positive (semi)definite matrix by filling in the missing $(i,j)$-th entries for $\{i,j\} \notin E$. There is a vast literature on this problem, starting with the seminal paper \cite{GJSW_1984}, which  answers the problem for chordal graphs. The case of cycles is settled in \cite{Barrett1993realpositivedefinite}, where positive (semi)definite completability is characterized using a cosine parametrization. The larger class  of graphs with no $K_4$-minors is treated in \cite{Laurent1997realpositivesemidefinite} and cycle completable graphs are covered in \cite{BJL-cycle-completability}. 
The main results are recalled in Section \ref{sec:prel-elliptope}.
Graph elliptopes are also used  within combinatorial optimization. Indeed, the set $\ME_n$ of correlation matrices defines a natural semidefinite programming relaxation for the maximum cut problem, which has been used by Goemans and Williamson \cite{GW_1995} for their celebrated $0.878$-approximation algorithm. In the same way, $\ME(G)$ provides a relaxation for the maximum cut problem on $G$. For  details see, e.g.,  \cite{Laurent_1997,Laurent-Poljak_1995} and references therein. 

The PSD matrix completion problem where one does not fix the diagonal entries to be $1$ is also widely investigated, see   \cite{GJSW_1984,Laurent_2001} and references therein.
This  leads to   the cone
 $\MCscr(G)\subseteq \R^{V\cup E}$ obtained by projecting the cone $\MS^n_+$ onto the subspace indexed by the vertices and edges of $G$. The elliptope $\ME(G)$ is an affine section of this cone.
There are tight links with the distance geometry problem and Euclidean distance matrix completion problems, see \cite{Laurent_1998,Laurent_2001}. 
In addition, the cone $\MCscr(G)$ is used in algebraic statistics, where  it is directly relevant to the cone of sufficient statistics for the undirected Gaussian graphical model of $G$, as explained in  \cite{Sturmfels2010MultivariateGaussians}. 
In this context, the existence of a maximum likelihood estimate is equivalent to the condition that the set $\pi_{V \cup E}^{-1}(\mathbb{S}) {\cap} {\cal S}^n_{+}$ is non-empty, where $\mathbb{S}$ is the sample covariance matrix. The close relation between $\MCscr (G)$ and $\ME(G)$ leads the authors of \cite{Sturmfels2010MultivariateGaussians} to consider the algebraic boundary of $\MCscr(C_n)$ in \cite[Section 4.2]{Sturmfels2010MultivariateGaussians}, invoking results from \cite{Barrett1993realpositivedefinite}.
In \cite{BS-2020} the authors investigate the quality of the outer approximation of the cone $\MCscr(G)$ obtained by requiring positive semidefiniteness of the fully specified submatrices supported on the cliques of $G$ (the conic analog of the set $S(G)$ in (\ref{eq:SG})); they give an explicit analysis for cycle completable graphs and their suspensions.

The problem of finding positive semidefinite matrix completions of low rank has been investigated in \cite{NLV-Fields,NLV-JCTB,LV-MP-2014}; this connects to low dimensional Euclidean embeddings \cite{Belk,Belk-Connelly} and falls within the broader context of low rank solutions to semidefinite programs (see, e.g., \cite{Pataki} and references therein). 
It is known that any point $x\in \ME(C_n)$ has a positive semidefinite completion of rank at most 3 and a characterization is given for the points admitting such a completion of rank at most $2$ (see \cite[Section 2.2]{NLV-Fields}, or \cite[Lemma 4.6]{BS-2020}). We come back to this  in Remark~\ref{rem:cosLn}.

Geometric results on graph elliptopes and their polar bodies are presented in \cite{Solus-Uhler-Yoshida}.
In particular, for a graph $G$ with no $K_5$-minor, a connection is established between the facets of its cut polytope and the extreme rays of the polar body of $\ME(G)$. For such graphs it is known that the cut polytope coincides with the metric polytope and thus its facets arise from edge inequalities and cycle inequalities, as recalled in Section \ref{sec:prel-elliptope}.

Finding the algebraic boundary $\partial_a S$ of a convex semialgebraic set $S$ is a standard task in convex algebraic geometry, see  \cite{sinn2015algebraicboundary} for  details and applications.  
There,  it is shown that the algebraic boundary of a full-dimensional convex semialgebraic set $S$ is a hypersurface, i.e., its irreducible components have codimension one (see Proposition \ref{lemma:algboundconvex}).
So, the hypersurface $\partial_a S$ can be seen as a nonlinear generalization of the facet hyperplane arrangement for a convex polytope.

As mentioned above, the cycle polynomial $\Gamma_n$ is an important ingredient for describing the algebraic boundary $\partial_a \ME(G)$ of the  elliptope $\ME(G)$, when $G$ is cycle completable. 
This polynomial turns out to be interesting in its own right and, in fact, we have dedicated Section~\ref{sec:cyclepolynomial} to proving several of its remarkable properties. This part of the project was facilitated by recent results in \cite{mascarin2025lissajousvarieties} on the above mentioned class of varieties called Lissajous varieties, of which the zero locus of $\Gamma_n$ is an example (see Section~\ref{sec:lissajous}).
Lissajous varieties arising from graphs are investigated in  \cite[Section 4]{mascarin2025lissajousvarieties}, where they are used for computing equilibria in Kuramoto dynamics.

\paragraph{Organization of the paper.}
Section \ref{sec:prelim} contains preliminaries on graph elliptopes and Lissajous varieties, and introduces notation used throughout the paper. Section \ref{sec:cyclepolynomial} proves our new results related to cycle polynomials. The main theorems provide a resultant construction (Theorem \ref{thm:resultant}) and the degree of the homogeneous cycle polynomial (Theorem~\ref{thm:degGammah}). Section \ref{sec:boundary} includes our main results on the algebraic boundary of $\ME(G)$. The main result is Theorem~\ref{thm:boundary-cycle-completable}, which characterizes the algebraic boundary of $\ME(G)$ for cycle completable graphs (as a consequence of a more general result). 
This allows us  to characterize the graph elliptopes that are spectrahedra (Corollary~\ref{cor:spectrahedron}). In addition, when $G$ is cycle completable, we derive the algebraic boundary of the cone $\MCscr(G)$ (Corollary \ref{cor:coneCG}) and the algebraic boundaries of $\MCscr(\nabla G)$, $\ME(\nabla G)$ of the suspension graph $\nabla G$, obtained by   adding a new node adjacent to all nodes of $G$ (Proposition \ref{prop:suspensionboundary}).

\section{Preliminaries} \label{sec:prelim}

\subsection{Preliminaries on elliptopes}\label{sec:prel-elliptope}

A \emph{correlation matrix} is a (real symmetric) positive semidefinite matrix whose diagonal entries are all equal to one. 
We denote by $\mathcal{E}_n$ the set of all $n \times n$ real correlation matrices.
Let $G=(V=[n],E)$ be a simple, undirected graph; it is sometimes convenient to let $V(G)$ and $E(G)$ denote   the vertex set and the edge set of $G$.
The  {\em elliptope} of $G$, denoted by $\mathcal{E}(G)$, is defined as the projection of $\mathcal{E}_n$ onto the subspace $\R^E$ indexed by the edge set of $G$, i.e.,
\[
\mathcal{E}(G) = \pi_E({\cal E}_n) = \{ x\in \R^E : \exists X\in \mathcal{E}_n \text{ such that } X_{ij}=x_{ij} \text{ for } \{i,j\}\in E\}.
\]
Throughout, $\pi_E$ denotes the projection onto the subspace indexed by the edge set $E$ of $G$. 
In particular, $\ME_n$ corresponds to the elliptope of the complete graph $K_n$ (after identifying a correlation matrix with the vector consisting of its off-diagonal entries). Deciding whether a vector $x\in\R^E$ belongs to $\ME(G)$ is an instance of the {\em positive semidefinite completion problem}, which asks to determine whether the partially filled matrix with ones on the diagonal and entries $x_e$ at the positions indexed by the edges of $G$ can be completed to a positive semidefinite matrix.
This problem has been widely studied, starting with the work of Grone et al.~\cite{GJSW_1984} which proves the following necessary condition: If  $x\in \ME(G)$, then 
\begin{align}\label{eq:clique}
x[K]\succeq 0 \text{ for all cliques } K \text{ of } G.
\end{align}
This is known as the {\em clique condition}. Recall that a {\em clique} of $G$ is a set $K\subseteq V$ of vertices that are pairwise connected by an edge in $G$. In \eqref{eq:clique}, $x[K]$ denotes the symmetric matrix indexed by $K$ with an all-ones diagonal and whose off-diagonal entries are provided by the entries of $x$. The vertices of each edge in $E$ form a clique, so the clique condition implies that $\ME(G)$ is contained in the hypercube  $[-1,1]^E$. Clearly, in (\ref{eq:clique}), it suffices to consider the maximal cliques of $G$. Let $\MK_G$ denote the set of maximal cliques of $G$.  
Grone et al.  \cite{GJSW_1984}  show that the clique condition (\ref{eq:clique}) is necessary \emph{and sufficient} for membership in $\ME(G)$ if and only if the graph $G$ is chordal. 

\begin{theorem}
\label{theo:chordal}\cite{GJSW_1984}
For a graph $G=(V,E)$, define the spectrahedron
\begin{align}\label{eq:SG}
S(G)\, =\, \{x\in\R^E: x[K]\succeq 0\ \text{ for all } K\in \MK_G\}.
\end{align}
We have the inclusion $\ME(G)\subseteq S(G)$, and equality holds if and only if $G$ is a chordal graph.
\end{theorem}

One of the many equivalent definitions of $G$ being chordal   is that $G$ does not contain a chordless cycle of length at least 4 (i.e., every cycle  of length at least 4 has a chord).  Here, a {\em cycle} $C$ in $G$ of length $p\ge 3$ is a subgraph of $G$ whose vertices  can be ordered as $v_1,\ldots, v_p$ in such a way that the pairs $\{v_1,v_2\},\ldots,\{v_{p-1},v_p\}, \{v_p,v_1\}$ are all edges in $G$; the cycle is then also denoted as $C=(v_1,\ldots,v_p)$. Any other edge of $G$ of the form $\{v_i,v_j\}$ with $j\ne i-1,i+1$ (modulo $p$) is called a {\em chord} of $C$, and $C$ is {\em chordless} if no such a chord exists. A cycle of length $n$ is denoted by $C_n$ and its edge set is denoted by $E_n$.

The inclusion $\ME(G)\subseteq S(G)$ is strict when $G=C_n=(1,\ldots,n)$ is a cycle with $n \ge 4$; indeed, the vector $x\in\R^{E_n}$ with entries $x_{12}=\ldots=x_{n-1,n}=1$ and $x_{1n}=-1$ satisfies (\ref{eq:clique}), but $x\not\in \ME(C_n)$  as $x$  is not completable to a positive semidefinite matrix.

 When $G$ contains  (chordless) cycles,  
another necessary condition
was introduced by Barrett et al.~\cite{Barrett1993realpositivedefinite}: If $x\in \ME(G)$, then it can be parametrized as $x=\cos (\pi a)$ for some $a\in [0,1]^E$ (since $x\in [-1,1]^E$), and the vector $a={1\over \pi} \arccos(x)$ satisfies 
\begin{align}\label{eq:cycle}
a(F)-a(C\setminus F)\le |F|-1 \text{ for all cycles } C \text{ in } G, F\subseteq C, |F| \text{ odd}.
\end{align}
These are called the \emph{cycle inequalities}. 
Here, for a subset $F\subseteq E$, we set $a(F)=\sum_{e\in F}a_e$.
Let $\MET(G)$ denote the polytope in $\R^E$ consisting of the vectors $a\in \R^E$ satisfying  $0\le a_e\le 1$ for $e\in E$ and the cycle inequalities \eqref{eq:cycle}. This polytope is known as the {\em metric polytope} of $G$. 
It is well-known that the inequalities $0\le a_e\le 1$ define facets of $\MET(G)$ only for the edges $e\in E$ that are not contained in a larger clique of $G$. Moreover,  only the chordless cycles are needed in (\ref{eq:cycle}) since these are precisely the ones that induce facets of $\MET(G)$ (see \cite{Barahona-Mahjoub}, or  \cite[Chap. 27]{Deza-Laurent-1997}).  Indeed, if the edge $e_0\in E$ is a chord of a cycle $C$ in $G$, then one can cover the edge set of $C$ with two smaller cycles (call them $C'$, $C''$) sharing only the edge $e_0$, and one can easily see that the cycle inequality (\ref{eq:cycle}) obtained from $C$ is implied by the cycle inequalities obtained from $C'$ and $C''$ by eliminating the variable $x_{e_0}$. 
This idea of eliminating the variable attached to a chord of a cycle will be used again in Section \ref{sec:resultant} 
when dealing with the resultant of cycle polynomials. 

The result of Barrett et al.~\cite{Barrett1993realpositivedefinite} can be expressed as
\begin{align}\label{eq:inclEMET}
\ME(G)\subseteq \cos(\pi \MET(G)),
\end{align}
where the cosine map is applied entrywise. 
The paper \cite{Barrett1993realpositivedefinite} also shows that equality holds in (\ref{eq:inclEMET}) when $G$ is a cycle. In \cite{Laurent1997realpositivesemidefinite}, Laurent characterizes all graphs for which equality holds.
 
\begin{theorem} \label{thm:cosineMETnoK4minors} \cite{Laurent1997realpositivesemidefinite}
For a graph $G$, equality $\ME(G)=\cos(\pi\MET(G))$ holds if and only if $G$ has no $K_4$-minor (equivalently, $G$ is series-parallel).
\end{theorem}

By the above discussion, 
the facets of $\MET(K_n)$ are defined by the cycle inequalities (\ref{eq:cycle}) for all triplets $\{i,j,k\}\subseteq V(K_n)$. These are also known as the {\em triangle inequalities}. Moreover, the metric polytope of a graph $G=(V,E)$ with $n$ vertices is obtained as a projection of $\MET(K_n)$: 
 \begin{align}\label{eq:METG}
 \MET(G)=\pi_E(\MET(K_n)),
 \end{align}
see \cite{Barahona-Mahjoub}, or  \cite[Chap. 27]{Deza-Laurent-1997}. The inclusion $\pi_E(\MET(K_n)) \subseteq \MET(G)$ is clear since each cycle inequality is implied by triangle inequalities.
The first row in Figure~\ref{fig:E4} shows three different sections of the metric polytope $\MET(C_4)$, obtained by setting a variable to $\tfrac{1}{2}, \tfrac{1}{3},0$, respectively.

As $C_3=K_3$, we have $\ME(K_3)=\cos (\pi\MET(K_3))$. Combining this with (\ref{eq:METG}), it follows that  $\cos(\pi\MET(G))$ is a spectrahedral shadow  for any graph $G$.

\begin{lemma}\label{lem:cosMET}
For any graph $G$, the set 
$\cos(\pi\MET(G))$ is the projection of a spectrahedron. In particular, $\cos(\pi\MET(G))$ is convex and semialgebraic.
\end{lemma}

\begin{proof}
Since   $\MET(K_3)$ is defined by the triangle inequalities that are based on three point-sets and the equality $\cos (\pi\MET(K_3)) =\ME(K_3)$ holds by Theorem \ref{thm:cosineMETnoK4minors}, one obtains that 
\[\cos(\pi\MET(K_n))=\{x\in \R^{E(K_n)}: x[K]\succeq 0 \text{ for all } K\subseteq  V, |K|=3\}.\]
 Hence, the set $ \cos(\pi\MET(K_n))$ is a spectrahedron, so it is convex and semialgebraic. Since the projection of a convex set is convex, and the projection of a semialgebraic set is semialgebraic by Tarski-Seidenberg,   the projection  $\pi_E(\cos(\pi\MET(K_n)))$ is convex semialgebraic too. Using \eqref{eq:METG}, we have $\pi_E(\cos(\pi\MET(K_n)))= \cos(\pi_E(\pi\MET(K_n)))= \cos(\pi\MET(G))$.\end{proof}

Theorem \ref{theo:chordal} and Equation \eqref{eq:inclEMET} provide  two convex semialgebraic sets  containing $\ME(G)$. We combine them to obtain a sharper outer approximation  $\tME(G)= S(G)\cap \cos (\pi\MET(G))$,  
\begin{align}\label{eq:tMEG}
\begin{split}
\tME(G)  
=
\Big\{x\in \R^E: x[K]\succeq 0 \text{ for all } K\in\MK_G  \text{ and } {1\over \pi}\arccos(x)\in\MET(G)\Big\},
\end{split}
\end{align}
  of the elliptope $\ME(G)$: we have the inclusion $\ME(G)\subseteq \tME(G)$. Notice that $\tME(G)$  is full-dimensional since the zero vector lies in its interior. It is convex and semialgebraic, as these properties are preserved when taking intersections. 
Following Barrett et al.  \cite{BJL-cycle-completability}, a graph $G$ is said to be {\em cycle completable}  when membership in the elliptope $\ME(G)$ is characterized by both conditions (\ref{eq:clique}) and (\ref{eq:cycle}) combined, i.e., when the equality $\ME(G)=\tME(G)$ holds.
The following characterization of cycle completable graphs follows from results in \cite{BJL-cycle-completability,Johnson-McKee}.

\begin{theorem}\cite{BJL-cycle-completability,Johnson-McKee}\label{theo:cycle-completable} 
A graph $G$ is cycle completable, i.e., equality $\ME(G)=\tME(G)$ holds with $\tME(G)$ as in (\ref{eq:tMEG}), 
if and only if $G$ satisfies any of the following equivalent conditions:
\begin{itemize}
\item[(i)] $G$ is a subgraph of a chordal graph, all of whose cliques of size $4$ are contained in $G$.
\item[(ii)] $G$ can be obtained as a clique sum of chordal graphs and graphs with no $K_4$-minors.
\item[(iii)] No induced subgraph of $G$ is a wheel $W_k$ ($k\ge 5$) or an iterated splitting of $W_k$ ($k\ge 4$).
\end{itemize}
\end{theorem}
We refer to \cite{BJL-cycle-completability,Johnson-McKee} for details on clique sums and splittings.
We will consider elliptopes of wheels later in Section \ref{sec:applications}.

\subsection{Preliminaries on Lissajous varieties}  \label{sec:lissajous} 

In this section we discuss Lissajous varieties, a class of   algebraic varieties that was recently introduced and investigated by Mascarin and Telen  
 \cite{mascarin2025lissajousvarieties}.
 
We refer to the textbooks \cite{CoxLittleOShea2005UsingAG,cox2025ideals} for basic facts about algebraic varieties and polynomial  ideals. In what follows, for a set $\MF\subseteq \C[x_1,\ldots,x_n]$ of polynomials, $V(\MF)=\{x\in\C^n: f(x)=0 \text{ for all } f\in\MF\}$ denotes the associated complex affine variety.  The \emph{vanishing ideal} of a set $S \subseteq \mathbb{C}^n$ is denoted by $I(S) \, = \, \{ f \in \mathbb{C}[x_1,\ldots, x_n] \, : \, f(x) = 0 \text{ for all } x \in S \}$. 

We use the notation $u\cdot v$ for the euclidean inner product of two vectors $u,v\in\C^n$, and $u\circ v=(u_iv_i)_{i=1}^n$ for their Hadamard product. We let $\ao$ and $\azero$ denote the all-ones vector and the zero vector (of appropriate size), and $I_n$ is the $n\times n$ identity matrix.

\subsubsection*{Parametric and determinantal representations of Lissajous varieties}

Let $A \in \mathbb{Z}^{d \times n}$ be an integer-valued  matrix and let ${\rm Row}(A) \subseteq \mathbb{C}^n$ denote the $\mathbb{C}$-linear span of its rows. 
Given a  vector $b \in \mathbb{C}^n$, let ${\cal L}_{A,b}$ be the image of the affine space 
\begin{align*}
L_{A,b} = {\rm Row}(A) -  \frac{b\pi}{2} =\Big \{ x -  \frac{b \pi}{2}\, : \, x \in {\rm Row}(A)\Big\} \subseteq \mathbb{C}^n
\end{align*}
 under the componentwise cosine map. That is, define
\begin{align*}
{\cal L}_{A,b} \, = \, {\rm cos}(L_{A,b}) =\{(\cos(z_1),\dotsc,\cos(z_n)): z\in L_{A,b}\}\subseteq \C^n,
\end{align*}
called the {\em Lissajous variety} associated to the pair $(A,b)$. Then, $\ML_{A,b}$  is an irreducible affine algebraic variety of dimension ${\rm rank}(A)$   \cite[Lemma 2.1]{mascarin2025lissajousvarieties}.
The case where $b= {\bf 0} \in \mathbb{C}^n$ first appeared in \cite[Section 5]{BelAfiaMeroniTelen2025Chebyshev}. We shall use the short-hand notation ${\cal L}_A = {\cal L}_{A,{\bf 0}}=\cos(\Row(A))$. 
 
Let $a_1,\ldots,a_n\in \Z^d$ denote the columns of matrix $A$. Choosing coordinates $t \in \mathbb{C}^d$ on ${\rm Row}(A)$,  
one can obtain   the Lissajous variety  
${\cal L}_{A,b}$ as   the image  of $\C^d$ under the map
 \begin{equation}\label{eq:parametrization} 
\phi_{A,b} \, : \, \mathbb{C}^d \longrightarrow \mathbb{C}^n, \quad t = (t_1, \ldots, t_d) \longmapsto ( {\rm cos}(a_1 \cdot t - b_1 \tfrac{\pi}{2}), \ldots, {\rm cos}(a_n \cdot t - b_n \tfrac{\pi}{2} )),
\end{equation}
which gives a {\em trigonometric parametrization} of $\ML_{A,b}=\phi_{A,b}(\C^d)$.

The integer matrix $A$ defines an affine toric variety $Y_A \subseteq \mathbb{C}^n$,
 obtained as the closure of the image of the Laurent monomial map $v \in (\mathbb{C}^*)^d\mapsto (v^{a_1}, \ldots, v^{a_n})$,   i.e., 
 \begin{align*}
 Y_A \, = \, \overline{\{v^{a_1},\ldots,v^{a_n}): v\in (\C^*)^d\}} \,  \subseteq \,  \C^n.
 \end{align*}
A \emph{scaled} toric variety $Y_{A,\beta}$ is obtained from $A \in \mathbb{Z}^{d \times n}$ and  $\beta \in (\mathbb{C}^*)^n$ as follows:
\[ Y_{A,\beta} \, = \, \{ (\beta_1 \, y_1, \ldots, \beta_n \, y_n) \, : \,  y \in Y_A \} \, \subseteq \, \mathbb{C}^n. \]
The vanishing ideal $I(Y_{A,\beta})$ is generated by binomials: 
\begin{align}\label{eq:IYA}
 I(Y_{A,\beta}) \, = \, \langle \beta^wy^{u} - \beta^uy^{w} \, : \, u, w \in \mathbb{N}^n \, \, \text{and} \, \,  A(u-w) = 0 \rangle \, \subseteq \, \C[y]=\mathbb{C}[y_1, \ldots, y_n].
 \end{align}
For instance, the vanishing ideal $I(Y_A)$ is obtained by setting  $\beta=\one$. The following characterization, shown in  \cite[Theorem 2.2]{mascarin2025lissajousvarieties}, immediately yields a {\em rational parametrization} of ${\cal L}_{A,b}$.

\begin{theorem}\label{thm:dimension}
Let  $A\in\Z^{d\times n}$, $b\in\C^n$, and set $\beta = (e^{-i b_1 \tfrac{\pi}{2}}, \ldots, e^{-i b_n \tfrac{\pi}{2}}) \in (\mathbb{C}^*)^n$.
Then, the    Lissajous variety ${\cal L}_{A,b} \subseteq \mathbb{C}^n$ is equal to the image of the variety
    \[ {\cal Y}_{A,b} \, = \, \{(x,y) \in \mathbb{C}^n \times (\mathbb{C}^*)^n \, : \, y \in Y_{A,\beta} \text{ and } x_j \, = \, \tfrac{1}{2}(y_j + y_j^{-1}) \text{ for } j = 1, \ldots,n \}, \]
    under the coordinate projection $\pi_{A,b}\colon (x,y)\in {\cal Y}_{A,b}  \mapsto x\in \mathbb{C}^n$. 
      That is, $$\ML_{A,b}=\pi_{A,b}({\cal Y}_{A,b}).$$
\end{theorem}

This characterization can be used to derive a {\em determinantal representation} of the Lissajous variety $\ML_{A,b}$, as shown in \cite[Theorem 3.1]{mascarin2025lissajousvarieties}. For simplicity, and since this is the relevant case for our paper, we present this determinantal representation in the case when $A$ has rank $n-1$ and $b=\mathbf{0}$, i.e., $\beta=\one$.
The kernel of $A$ is spanned by a  vector $u-w$ with $u,w\in\N^d$ and
$I(Y_A)=\langle g_A\rangle$,   where $g_A=y^u-y^w\in \C[y]$. Given $x\in\C^n$, consider the ideal 
$$J_x=\langle 2x_i-y_i-y_i^{-1}: i\in [n]\rangle\subseteq \C[y,y^{-1}]=\C[y_i,y_i^{-1}: i\in [n]].$$
 As the equation $2x_i-y_i-y_i^{-1}=0$ is equivalent to $y_i^2-2x_iy_i+1=0$,  the variety $V(J_x)$ is a finite set, consisting of  $2^n$ points, 
\begin{align}\label{eq:VJx}
V(J_x)=\Big\{\big(x_i+\epsilon_i \sqrt{x_i^2-1}\big)_{i=1}^n: \epsilon\in \{\pm 1\}^n\Big\}.
\end{align}
Consider the quotient algebra 
$R_x= \C[y,y^{-1}]/J_x$. Hence, $R_x\simeq \otimes_{i=1}^n \C[y_i]/\langle y_i^2-2x_iy_i+1\rangle$ has finite dimension $2^n$, since $\C[y_i]/\langle y_i^2-2x_iy_i+1\rangle$ is spanned by the cosets of  $1,y_i$. 
 For any polynomial $g\in \C[y]$, consider the associated ``multiplication by $g$'' linear map $M_g: R_x \to R_x$ that maps the coset of $f$ to the coset of $fg$.  It follows from the well-known eigenvalue method for polynomial root finding (see, e.g., \cite[Chapter 4, Proposition 2.7]{CoxLittleOShea2005UsingAG})  that the eigenvalues of $M_g$ are the evaluations $g(\xi)$ of $g$ at the roots $\xi\in V(J_x)$. In particular, we have
\begin{align}\label{eq:detg}
\det (M_g)= \prod_{\xi\in V(J_x)} g(\xi).
\end{align}
The matrix $M_g$ is obtained as  $M_g=g(M_{y_1},\ldots,M_{y_n})$, a polynomial in the pairwise commuting matrices $M_{y_i}$. In \cite[Section 3]{mascarin2025lissajousvarieties}, it is explained in detail that the matrix $M_{y_i}$ is the following Kronecker product of $n$ matrices of size $2 \times 2$:
\[  
I_2 \otimes \cdots \otimes I_2\otimes \left ( \begin{smallmatrix} 0 & -1\cr 1 & 2x_i\end{smallmatrix} \right) 
 \otimes I_2\otimes \cdots \otimes I_2,\]  
where the only non-identity matrix appearing in the product is placed at the $i$-th position.  
Hence, $M_g$ is a $2^n\times 2^n$ matrix with polynomial entries in $\Z[x]$.

By Theorem \ref{thm:dimension}, $x\in\ML_{A,b}$ if and only if there exists $y\in V(J_x)$ such that $y\in Y_{A}=V(g_A)$; equivalently, $x\in\ML_{A,b}$ if and only if $\prod_{y \in V(J_x)}g_A(y)=0$. Therefore, using (\ref{eq:detg}) for the polynomial $g = g_A$, we obtain the following {\em determinantal representation} of $\ML_{A,b}$,
\begin{align}\label{eq:detrep}
\ML_{A,b}=\{x\in\C^n: \det(M_{g_A})=0\}.
\end{align}

\subsubsection*{Lissajous varieties of graphs}
An important class of Lissajous varieties arises from graphs in the following way.
 Let $G=(V,E)$ be a simple, connected graph. 
 We define an {incidence matrix} $A_G\in \mathbb{Z}^{V\times E}$ of $G$ as follows. Select an ordering and orientation of the edges in  $E$, and an ordering of the vertices in $V$. The rows of $A_G$ are indexed by $V$ and its columns are indexed by $E$.
If $(j,k)$ is the $l$-th edge of $G$ (oriented from $j$ to $k$),
  then the $l$-th column of $A_G$ is the vector $a_l = e_j-e_k$, where $e_k$ is the $k$-th standard basis vector of $\mathbb{R}^V$.
That is, $(A_G)_{i,(j,k)}$ is equal to 1 (resp., $-1$, 0) if $i=j$ (resp., $i=k$, $i\in V\setminus\{j,k\}$).
 Since $G$ is connected, the rank of $A_G$ is $|V|-1$. 
The  Lissajous variety associated to $G$  
is
\begin{align*}
   \ML_G  
  =\ML_{A_G,{\bf 0}}=\cos (\Row(A_G)).
\end{align*}
Reordering the edges simply amounts to permuting the coordinates in the ambient space~$\mathbb{C}^n$ and  $\mathcal L_{G}$ does not depend on the selected orientation of the edges. Indeed, reversing the edge $(j,k)$ amounts to replacing the associated column $a_l=e_j-e_k$ of $A_G$ by its negation  $-a_l=e_k-e_j$; as $\cos(-a_l\cdot t)=\cos(a_l\cdot t)$, the variety $\mathcal L_{G}$ remains unchanged.

The Lissajous variety of the cycle graph $G=C_n=(1,2,\ldots,n)$ plays a central role in this paper. For a cycle of length $n$, we simplify the notation further to
$$\mathcal L_n=\ML_{C_n}.$$ 
Let us order and orient the edges of $C_n$ as
 $(1,n), (1,2),\ldots, (n-1,n)$. 
 Since $A_{C_n}$ has rank $n-1$ and the Lissajous variety $\mathcal L_n$ depends only on the row span of $A_{C_n}$, 
one can delete a row of $A_{C_n}$ and apply row transformations so that its row space   is equal to that of  \begin{equation} \label{eq:cycleAn} A_{n} \, = \, \begin{pmatrix}
    1 & 1 & 0 & \cdots & 0 \\ 
    1 & 0 & 1 & \cdots & 0 \\ 
    \vdots & \vdots & \vdots & \ddots & \vdots \\ 
    1 & 0 & 0 & \cdots & 1
\end{pmatrix} \quad \in \, \, \mathbb{Z}^{(n-1) \times n}.\end{equation}
 The associated Lissajous variety $\mathcal L_n \subseteq \mathbb{C}^n$ is 
 \begin{equation}
\label{eq:parametrize}
\mathcal L_n \,  = \, \cos(\Row(A_n)) \, = \, \{( \cos(t_2+\dotsc+t_n), \cos(t_2),\dotsc, \cos(t_n )): t\in \C^{n-1}\}.
\end{equation}

A Lissajous variety depends only on the linear space ${\rm Row}(A) \subseteq \mathbb{C}^n$, and not on the particular choice of the matrix $A \in \mathbb{Z}^{d \times n}$. 
This correspondence is still not one-to-one, as multiple affine  spaces may define the same Lissajous variety. 
The following result provides such an observation, which we will use for the case of ${\cal L}_n$.

\begin{lemma}\label{lem:cosH}
Consider a matrix $A\in\Z^{d\times n}$ of rank $d$. 
Let $\{v_j\}_{j=1}^{n-d}\subseteq\Z^n$ be a basis of the lattice $\ker(A)\cap \Z^n = \{z \in \mathbb{Z}^n \, : \, A \, z = 0 \}$.
For the Lissajous variety ${\cal L}_A=\cos (\Row(A))$, we have $\ML_A=\cos(H_{s,k})$ for any affine space 
$H_{s,k}\subseteq \C^n$ of the form 
\[
H_{s,k}=\{t\in \C^n: (s\circ v_j)\cdot t= 2\pi k_j \text{ for all }j=1,\dotsc,n-d\},
\]
where $s\in \{\pm 1\}^n$ and $k\in \Z^{n-d}$.
Moreover, we have the equality 
\[
\cos^{-1}(\ML_A) = \bigcup_{s\in \{\pm 1\}^n,k\in \Z^{n-d}} H_{s,k}.
\]
\end{lemma}

\begin{proof}
Let $s\in\{\pm 1\}^n$ and $k\in \Z^{n-d}$. 
Assume that $x\in \cos(H_{s,k})$, i.e.,  $x=\cos (t)$  for some $t\in\C^n$ with $(s\circ v_j)\cdot t=2k_j\pi$ for all $j=1,\dotsc,n-d$. Choose $z\in\Z^n$ such that $k_j= v_j\cdot z$ for $j=1,\ldots,n-d$.
 We perform the linear change of variables 
\[
t_i  \mapsto t'_i= s_it_i -2z_i\pi,\text{ for } i=1,\dotsc, n.
\]
Then, we obtain $v_j\cdot t'= (v_j\circ s)\cdot t- 2\pi v_j\cdot z  = 0$ for $1\le j\le n-d$, showing $t'\in \Row (A)$. Then, $\cos(t')=\cos(t)=x$, and thus $x\in {\cal L}_A$.
The inclusion $\ML_A\subseteq \cos(H_{s,k})$ is proved analogously, by inverting the change of variables.

It remains to prove that  $\cos^{-1}(\ML_A) \subseteq \bigcup_{s,k} H_{s,k}$.
Let $t\in \C^n$ be such that $\cos(t)\in\ML_A$. 
By definition, we have $\cos(t)=\cos(t')$, for some $t' \in \Row(A)$. This means that there exist some $s\in \{\pm 1\}^n$ and $k'\in \Z^n$ such that $t=s\circ t'+2k'\pi$. 
In particular, $(s\circ v_j)\cdot t=2\pi (s\circ v_j)\cdot k'$ for $j=1,\dotsc,n-d$. Set $k=(v_1\cdot (s\circ k'),\dotsc,v_{n-d}\cdot (s\circ k'))$. Then
 $(s\circ v_j)\cdot t=2\pi k_j$ for $j=1,\dotsc,n-d$, which shows    that $t\in H_{s, k}$.
\end{proof}

\begin{remark}\label{rem:cosLn}  
As an application of Lemma \ref{lem:cosH}, the Lissajous variety of the cycle  $C_n$ is given by
\[
\ML_n=\cos(H_{s,k}), \text{ where } H_{s,k} = \Big\{t\in \C^n: s\cdot t=2k\pi \Big\},\  s\in\{ \pm 1\}^n \text{ and } \ k\in \Z,
\]
and we have 
\begin{align}\label{eq:cosLn}
 \cos^{-1}(\ML_n)=\bigcup_{s\in\{\pm 1\}^n,k\in\Z}H_{s,k}=\{t\in \C^n: s\cdot t=2k\pi \}.
\end{align}
In particular, $\ML_n=\ML_{A'_n}$, where $A'_n$ is obtained from $A_n$ in (\ref{eq:cycleAn}) by changing the signs in its first column.
Indeed, the kernel of $A_n$ is spanned by $v=(1,-1,\ldots,-1)$, and  selecting $s=\ao$, $k=0$ gives the hyperplane  $H_{{\bf 1},0} =\ker (A'_n) = \{ t \in \mathbb{C}^n\, : \, t_1 + \cdots + t_n = 0 \}$.

The intersection $ \ML_n\cap [-1,1]^n$ is particularly relevant to our treatment as well as its preimage ${1\over \pi}\cos^{-1}(\ML_n)\cap [0,1]^n$ under the componentwise cosine map, which can in fact be covered by a union of {\em finitely many} hyperplanes since one can restrict  to $|k|\le n/2$ in (\ref{eq:cosLn}).

Let us comment briefly on the relation to elliptopes. A point $x\in\ME(C_n)$ has a PSD completion of rank at most 2 
if and only if $\arccos x$ lies in the   union of hyperplanes in  (\ref{eq:cosLn}) (see \cite[Section 2.2]{NLV-Fields}, or \cite[Lemma 4.6]{BS-2020}), which is equivalent to $x\in \ML_n$.  Thus, the hypersurface ${\cal L}_n$ is the Zariski closure of the part of the boundary of $\ME(C_n)$ consisting of the points that admit a PSD completion of rank at most $2$. On the other hand,  the Zariski closure of the full boundary of $\ME(C_n)$ consists of $\ML_n$ together with the coordinate 
hyperplanes $x_e=\pm 1$ for $e\in E_n$ (see Proposition \ref{proposition:VPn}).
\end{remark}

\section{Cycle polynomials}\label{sec:cyclepolynomial}

In \cite[Section 4.2]{Sturmfels2010MultivariateGaussians} the authors define the $n$-th \emph{cycle polynomial} $\Gamma_n$ as the unique (up to scaling) $n$-variate irreducible polynomial obtained by rationalizing the equation
\begin{equation} \label{eq:rationalized equation}
    x_1 = \cos\Big(\sum_{i=2}^n \arccos(x_i)\Big).
\end{equation}
Here, it is meant that $\Gamma_n$ is the defining polynomial of the hypersurface obtained by applying the coordinatewise cosine map to the hyperplane $t_1 = t_2 + \cdots + t_n$. As this hyperplane is the row span of the matrix $A_n$ from \eqref{eq:cycleAn} this means  that  $V(\Gamma_n)=\cos(\Row(A_n))=\ML_n$.

The identity $V(\Gamma_n) = {\cal L}_n$ determines $\Gamma_n$ up to scaling. We shall now present a new definition (Definition \ref{def:cyclepolynomial}), consistent with $V(\Gamma_n) = {\cal L}_n$, which uniquely fixes this scaling. 

Let $C_n=(1,\ldots,n)$ be the cycle graph with $n$ vertices and $E_n=E(C_n)$ be its set of edges, denoted $\{1,2\}, \ldots, \{n-1,n\}, \{1,n\}$. It is convenient to use variables indexed by the edges of  $C_n$, so $\Gamma_n\in \C[x_e:e\in E_n]$. When $C$ is a cycle in a graph $G=(V,E)$, we also denote the corresponding cycle polynomial as $\Gamma_C\in \R[x_e: e\in E(C)]$.  

It is convenient  to use the alternative, equivalent definition of the Lissajous variety $\ML_n$ as $\ML_n=\ML_{A'_n}$,
where $A'_n$ is obtained by changing the signs in the first column of $A_n$. 
The matrix $A'_n$ defines a \emph{different} toric variety $Y_{A_n'}$ but the \emph{same} Lissajous variety ${\cal L}_{A_n} = {\cal L}_{A_n'} = {\cal L}_n$ (recall Remark \ref{rem:cosLn}). 
The kernel of $A_n'$ is spanned by the all-ones vector, which has the advantage that    the defining equation of $Y_{A_n'}$ is more symmetric than that of $Y_{A_n}$. Concretely, as explained in Section \ref{sec:lissajous}, $Y_{A_n'}=V(g_n)$, where  $g_n(y) \, = \, \prod_{e\in E_n} y_e - 1$. 
Using relations (\ref{eq:detg}) and  (\ref{eq:detrep}), we obtain 
\[
\ML_n=\Big \{x\in\C^{E_n}: \prod_{\xi\in V(J_x)} g_n(\xi)=0\Big \}=\left\{x\in \C^{E_n}: h_n(x)=0\right\},
\]
where $V(J_x)$ is as in (\ref{eq:VJx}) and   the function $h_n$ is defined as 
\begin{equation}\label{eq:hn}
h_n(x) \,  = \, \prod_{\xi \in V(J_x)} g_n(\xi) \, = \, \prod_{\epsilon \in \{-1,1\}^{E_n}} \Big[ \prod_{e \in E_n} (x_e + \epsilon_e \sqrt{x_e^2-1} ) - 1 \Big].
\end{equation}
By the discussion in Section \ref{sec:lissajous},  $h_n=\det(M_{g_n})$ is a polynomial in $\Z[x]$.
It follows from results in \cite{mascarin2025lissajousvarieties} that  $h_n$ is, in fact, the square of a polynomial.

\begin{theorem} \label{lem:existence}
    For $n \geq 3$, the polynomial  $h_n$ from (\ref{eq:hn}) 
    is the square of an irreducible polynomial and the  degree of $h_n$  is equal to
   $n \binom{n-1}{\lfloor (n-1)/2 \rfloor}$.
\end{theorem}

\begin{proof}
By \cite[Proposition 3.5]{mascarin2025lissajousvarieties}, the determinant of $M_{g_n}$ is $f^w$, where $f\in\C[x]$ is a defining equation of ${\cal L}_n$ and $w$ is a positive integer (the degree of the map $\pi_{A,b}$ from Theorem \ref{thm:dimension}). For the  Lissajous hypersurface $\ML_n$, this integer is shown to be equal to $2$ in the proof of \cite[Proposition 2.6]{mascarin2025lissajousvarieties}. The same proposition states the degree of $f$, which is $1/2$ times the stated degree of $h_n$. 
\end{proof}

\begin{definition} \label{def:cyclepolynomial}
Let  $h_n$ be as in Theorem \ref{lem:existence}. The \emph{$n$-th cycle polynomial} $\Gamma_n$
 is    the unique  polynomial in  $\mathbb{C}[x_e: e \in E_n]$  that  satisfies $2^{-2^{n-1}} h_n = \Gamma_n^2$ and whose leading coefficient is equal to 1 when viewed as a univariate polynomial in any  variable $x_{e_0}$. 
\end{definition}

In Proposition \ref{prop:univleadingterm} we will show  that,   when viewing $h_n$ as a univariate polynomial in a variable $x_{e_0}$, its leading coefficient is $2^{2^{n-1}} $, which motivates the scaling in the above definition.

\begin{example} \label{ex:smallGamma}
We mention the cycle polynomials $\Gamma_n$ for small $n$ and will come back later to them in Examples \ref{ex:Gamma3} and \ref{ex:Gamma4}. 
    For $n = 3$, the polynomial $\Gamma_3$ reads  
    \[ \Gamma_3 \, = \, -1 + x_{12}^2 + x_{23}^2 + x_{13}^2 - 2\, x_{12}x_{23}x_{13}
    =-\det\left(\begin{smallmatrix} 1 & x_{12}& x_{13}\cr x_{12} & 1 & x_{23}\cr x_{13} & x_{23}& 1 \end{smallmatrix}\right), \] 
    hence it  defines the Cayley cubic surface.
        Its square is a product of eight factors, each corresponding to a sign vector $\epsilon=(\epsilon_{12},\epsilon_{23},\epsilon_{13})\in\{\pm 1\}^{E_3}$:
    \begin{align*}
        h_3 \, = \, 16 \cdot \Gamma_3^2 & \, = \, \prod_{\epsilon\in \{\pm 1\}^{E_3}} \,  \Big(x_{12} +\epsilon_{12}  \sqrt{x_{12}^2-1}\Big)\Big(x_{23} +\epsilon_{23}  \sqrt{x_{23}^2-1}\Big)\Big(x_{13}  +\epsilon_{13}  \sqrt{x_{13}^2-1}-1\Big).
    \end{align*}
The 4-th cycle polynomial $\Gamma_4$ is given by 
    \[ \begin{matrix} \Gamma_4 \, = \,  x_{12}^4 - 4 x_{12}^3 x_{23} x_{34} x_{14} + 4 x_{12}^2 x_{23}^2 x_{34}^2 + 4 x_{12}^2 x_{23}^2 x_{14}^2 - 2 x_{12}^2 x_{23}^2 + 4 x_{12}^2 x_{34}^2 x_{14}^2  - 2 x_{12}^2 x_{34}^2 \\ \quad - 2 x_{12}^2 x_{14}^2 - 4 x_{12} x_{23}^3 x_{34} x_{14} - 4 x_{12} x_{23} x_{34}^3 x_{14} - 4 x_{12} x_{23} x_{34} x_{14}^3 + 8 x_{12} x_{23} x_{34} x_{14} + x_{23}^4 \\ + 4 x_{23}^2 x_{34}^2 x_{14}^2 - 2 x_{23}^2 x_{34}^2 - 2 x_{23}^2 x_{14}^2 + x_{34}^4 - 2 x_{34}^2 x_{14}^2 + x_{14}^4. \end{matrix} \]
    Finally, $\Gamma_5$ has degree 15 and it has 339 terms. We display a few of them for concreteness: 
    \[ \Gamma_5 \, = \, x_{12}^8 - 8 x_{12}^7 x_{23} x_{34} x_{45} x_{15} + 16 x_{12}^6 x_{23}^2 x_{34}^2 x_{45}^2 + 16 x_{12}^6 x_{23}^2 x_{34}^2 x_{15}^2 -  \cdots - 4 x_{15}^2 + 1.\]
\end{example}

\subsection{Properties  of the cycle polynomial}
A first observation is that $h_n$ and $\Gamma_n$ enjoy   the following symmetry properties. 

\begin{lemma} \label{lem:symmetry}
    The polynomial $h_n$ is symmetric, i.e., it is invariant under any permutation of the variables.
    Moreover, if $\tilde{x}$ is obtained from $x \in \mathbb{C}^{E_n}$ by an even sign flip, i.e., by replacing an even number of entries $x_e$ by $\tilde{x}_e = -x_e$, then $h_n(x) = h_n(\tilde{x})$ holds. 
    The same properties hold for the cycle polynomial $\Gamma_n$. 
\end{lemma}

\begin{proof}
This follows easily using the definition of $h_n$ in (\ref{eq:hn}) and of $\Gamma_n$ in Definition \ref{def:cyclepolynomial}.
 \end{proof}

Hence, contrary to what the name suggests, the cycle polynomial has more than only cyclic symmetry. This is also reflected in the symmetry of the row span of $A_n'$ and of the toric variety $Y_{A_n'}$. 
 The next result follows directly from the invariance under even sign flips.
 
\begin{corollary} \label{cor:even}
    Let $E' \subseteq E$ be a subset of edges with even cardinality. The cycle polynomial $\Gamma_n$, viewed as a polynomial in $x_e$ for $e \in E'$, is even. In other words, if we expand
    \begin{align}\label{eq:cycle-even}
     \Gamma_n \, = \, \sum_{\alpha \in \mathbb{N}^{E'}} c_\alpha(x_{E \setminus E'}) \, \prod_{e \in E'} x_e^{\alpha_e}, 
     \end{align}
    where the coefficients $c_\alpha$ are polynomials in the variables $x_{E \setminus E'} = (x_e \, : \, e \in E \setminus E')$, then $\sum_{e \in E'} \alpha_e$ is even for each $\alpha$ such that $c_\alpha \neq 0$. 
\end{corollary}

To simplify the exposition, we introduce some notation that we will also use in the next sections.
Recall that $E_n=E(C_n)$.  
We define the following algebraic functions of $x$:
\begin{align}\label{eq:zepsilon}
z_\epsilon= \prod_{e\in E_n} \left( x_e + \epsilon_e \sqrt{x_e^2-1} \right)^{|\epsilon_e|} \ \text{ for any } \epsilon \in \{\pm1, 0\}^{E_n}.
\end{align}
For convenience we may also index $z$ by $\{\pm,0\}$ instead of $\{\pm 1,0\}$, e.g., writing $z_{+-0}$ for $z_{1,-1,0}$. 
Note that $z_\epsilon \cdot z_{-\epsilon} =1$ and $z_{\epsilon}=\prod_{i=1}^n z_{\azero, \epsilon_i, \azero}$, where $\epsilon_i$ is placed at the $i$-th position and $\azero$ denotes a zero vector of appropriate length. 
If $\delta$, $\delta'$ are sequences whose concatenation $(\delta,\delta')$ lies in $\{\pm 1,0\}^{E_n}$, then $z_{\delta,\delta'}=z_{\delta,\azero }\cdot z_{\azero, \delta'}$ holds. 
Using  (\ref{eq:zepsilon}), we can rewrite $h_n$ from (\ref{eq:hn}) compactly~as
\begin{align}\label{eq:hnz}
h_n=\prod_{\epsilon\in \{\pm 1\}^{E_n}} (z_\epsilon -1).
\end{align}

The next proposition motivates the factor $2^{-2^{n-1}}$ in Definition \ref{def:cyclepolynomial}: this scaling makes $\Gamma_n$ monic when viewed as a univariate polynomial in any of its variables.  

\begin{proposition} \label{prop:univleadingterm}
 Consider two distinct edges $e_0,e_1\in E_n$. When viewed as a univariate polynomial in $x_{e_0}$, the polynomial $h_n$ 
has degree $2^{n-1}$ with leading coefficient $2^{2^{n-1}}$, and its roots are given by the following algebraic functions of $x_e$ (for $e \in E_n \setminus \{e_0\}$):
\begin{align}\label{eq:rdelta}
 r_\delta \,  = \, \frac{z_{0,\delta} + z_{0,-\delta}}{2}  \quad \text{ for } \delta \in \{-1,1\}^{E_n \setminus \{e_0\}}, \end{align}
where $z_{0,\delta} \, = \, \prod_{e \in E \setminus \{e_0\}} (x_{e} + \delta_{e} \, \sqrt{x_{e}^2-1} )$ as in (\ref{eq:zepsilon}), with the first zero entry in the subscript corresponding to the variable $x_{e_0}$.
  Moreover,  we have
 \begin{align}\label{eq:hnroot}
 h_n= 2^{2^{n-1}} \prod_{\delta\in \{\pm 1\}^{E_n\setminus \{e_0\}}} z_{0,\delta} (x_{e_0}-r_\delta)
 =2^{2^{n-1}} \prod_{\delta\in \{\pm 1\}^{E_n\setminus \{e_0\}}: \delta_{e_1}=1}  (x_{e_0}-r_\delta)^2.
     \end{align}
    \end{proposition}

\begin{proof}
Starting  with the definition of $h_n$ in (\ref{eq:hnz}), we obtain
\begin{align*}
h_n(x)& =\prod_{\epsilon \in \{-1,1\}^{E_n}} (z_\epsilon -1)
= \prod_{\delta\in\{\pm 1\}^{E_n\setminus \{e_0\}}} \prod_{\epsilon_0\in \{\pm 1\}}
(z_{\epsilon_0,\delta}-1)\\
& =\prod_{\delta\in\{\pm 1\}^{E_n\setminus \{e_0\}}} \prod_{\epsilon_0\in \{\pm 1\}}(z_{\epsilon_0,\azero} z_{0,\delta}-1)
 =  \prod_{\delta\in\{\pm 1\}^{E_n\setminus \{e_0\}}}
(z_{+\azero}z_{0,\delta}-1)\cdot (z_{-\azero}z_{0,\delta}-1)\\
& =  \prod_{\delta\in\{\pm 1\}^{E_n\setminus \{e_0\}}}
(z_{0,\delta}^2-z_{0,\delta}(z_{+\azero}+z_{-\azero})+1)
= \prod_{\delta\in\{\pm 1\}^{E_n\setminus \{e_0\}}}
(z_{0,\delta}^2 
-2x_{e_0}   z_{0,\delta}
+1)\\
&=  \prod_{\delta\in\{\pm 1\}^{E_n\setminus \{e_0\}}} \Big[2 z_{0,\delta}\cdot (x_{e_0}-r_\delta)\Big].
\end{align*}
In this derivation, we used the fact that $z_{+\azero} z_{-\azero}=1$, $z_{+\azero}+z_{-\azero}=2x_{e_0}$,
and the identities $z_{0,\delta}^2-2x_{e_0}z_{0,\delta}+1 =-2z_{0,\delta} (x_{e_0}-{1\over 2}(z_{0,\delta}+ z_{0,\delta}^{-1}))
= -2z_{0,\delta} (x_{e_0}-r_\delta)$.
This gives the first representation of $h_n$ in (\ref{eq:hnroot}). 
To get the second one, consider the partition of $\{\pm 1\}^{E_n\setminus\{e_0\}}$ as $ \Delta\cup (-\Delta)$, where 
$\Delta=\{\delta\in \{\pm 1\}^{E_n\setminus\{e_0\}}:\delta_{e_1}=1\}$ and  $-\Delta=\{-\delta:\delta\in\Delta\}$.
Then, using  the fact that $z_{0,\delta} z_{0,-\delta}=1$ and $r_\delta=r_{-\delta}$, we obtain
\[
\prod_{\delta\in \{\pm 1\}^{E_n\setminus \{e_0\}}} z_{0,\delta} (x_{e_0}-r_\delta)
 =\prod_{\delta\in \Delta} z_{0,\delta} (x_{e_0}-r_\delta) \cdot z_{0,-\delta} (x_{e_0}-r_{-\delta})
 = \prod_{\delta\in \Delta} (x_{e_0}-r_\delta)^2,
 \]
which settles  (\ref{eq:hnroot}). So, the degree of $h_n$ in $x_{e_0}$ is $2^{n-1}$, with leading coefficient $2^{2^{n-1}}$.
\end{proof}

Proposition \ref{prop:univleadingterm} has the following implication for the cycle polynomial $\Gamma_n$ (Definition \ref{def:cyclepolynomial}).

\begin{corollary} \label{cor:univleadingterm}
  Consider two distinct edges $e_0,e_1\in E_n$. Let $r_\delta$ be as in (\ref{eq:rdelta}). When viewed as a univariate polynomial in $x_{e_0}$, $\Gamma_n$ is a monic polynomial of degree $2^{n-2}$, and it reads
  \begin{align}\label{eq:Gammauniv}
\Gamma_n= \prod_{\delta\in \{\pm 1\}^{E_n\setminus \{e_0\}}: \delta_{e_1}=1} (x_{e_0}-r_\delta).
\end{align}
\end{corollary}

\begin{example}\label{ex:Gamma3}
Let $n = 3$. Let us order the edges in $E_3=E(C_3)$ as $\{1,2\},\{2,3\},\{1,3\}$ and set $e_0=\{1,2\}$, $e_1=\{2,3\}$.
Equation (\ref{eq:Gammauniv}) in this example reads
\[
\Gamma_3=(x_{12}-r_{++})(x_{12}-r_{-+})= x_{12}^2-x_{12}(r_{++}+r_{-+})+r_{++}r_{-+}.
\]
One can easily verify that  
 $r_{++}+r_{-+}= {1\over 2}(z_{++}+z_{--}+z_{-+}+z_{+-})=2 x_{13}x_{23}$ and that 
 $r_{++}r_{-+}={1\over 4}(z_{++}+z_{--})(z_{-+}+z_{+-})=  x_{13}^2+x_{23}^2-1$.
 Hence we get,  as mentioned Example \ref{ex:smallGamma},
    \begin{align*}
   \Gamma_3= x_{12}^2 -2 x_{12}x_{13}x_{23}+ x_{13}^2+x_{23}^2-1.
  \end{align*}
\end{example}

As an application of Corollary \ref{cor:univleadingterm}, the Lissajous variety $\ML_n=V(\Gamma_n)$ does not contain any   affine subspace obtained by fixing a strict subset of coordinates. 
This will be useful in Section~\ref{sec:proof-prop}.

\begin{lemma} \label{lem:specializeGamma}
Consider a strict subset $J \subsetneq [n]$ and $z=(z_j)_{j\in J}\in\C^J$.
Define the $(n-|J|)$-variate polynomial $f\in \C[x_i: i\in [n]\setminus J]$, obtained by setting $x_j=z_j$ for $j\in J$, in the cycle polynomial $\Gamma_n(x)$. Then, the polynomial $f$  is not identically zero.
\end{lemma}
\begin{proof}
This follows from Corollary \ref{cor:univleadingterm}, which states that $\Gamma_n$ is monic when viewed as a univariate polynomial in any of its variables $x_i$. In particular, we may pick $i\in [n]\setminus J$.\end{proof}

Next, we compute the degree of $\Gamma_n$ as a bivariate polynomial in two selected variables, since we will need this result later in Section \ref{sec:cyclehomogeneous}.

\begin{proposition} \label{prop:bivariatedegree}
    Viewed as a bivariate polynomial in any two distinct  edge variables $x_{e_0}, x_{e_1}$, the cycle polynomial $\Gamma_n$ has degree $2^{n-2}$. 
\end{proposition}

\begin{proof}
By Lemma \ref{lem:symmetry}, it suffices to show that the   degree of $h_n$ as a bivariate polynomial in $x_{e_0}, x_{e_1}$ is equal to $2^{n-1}$.
The proof is analogous to that of Proposition \ref{prop:univleadingterm}, only a bit more technical.
Set $F=E_n\setminus \{e_0,e_1\}$ and assume that $e_0,e_1$ come first and second in the ordering of $E_n$ used in the subscript of $z_\epsilon$.
We start from the definition (\ref{eq:hnz})  for $h_n$, which we now rewrite as 
\begin{align*}
h_n & = \prod_{\delta \in \{\pm 1\}^F} \prod_{\epsilon_0,\epsilon_1\in\{\pm 1\}} 
(z_{\epsilon_0,\epsilon_1,\delta}-1) 
= \prod_{\delta \in \{\pm 1\}^F} \prod_{\epsilon_0,\epsilon_1\in\{\pm 1\}}
(z_{\epsilon_0,\epsilon_1,\azero}z_{0,0,\delta}-1)\\
&= \prod_{\delta \in \{\pm 1\}^F} 
(z_{++\azero}z_{0,0,\delta}-1)(z_{+-\azero}z_{0,0,\delta}-1)(z_{-+\azero}z_{0,0,\delta}-1)(z_{--\azero}z_{0,0,\delta}-1)\\
&= \prod_{\delta \in \{\pm 1\}^F} 
[(z_{0,0,\delta})^4 \sigma_4-(z_{0,0,\delta})^3\sigma_3+(z_{0,0,\delta})^2\sigma_2 -z_{0,0,\delta}\sigma_1+1],
\end{align*}
  where $\sigma_k$ is the sum of the products of $k$ factors among $z_{++\azero}, z_{+-\azero}, z_{-+\azero}, z_{--\azero}$ for $k=1,2,3,4$. We have 
  $\sigma_4=z_{++\azero} z_{+-\azero} z_{-+\azero} z_{--\azero}=1$, $\sigma_3=\sigma_1=z_{++\azero}+ z_{+-\azero}+ z_{-+\azero}+ z_{--\azero}= 4x_{e_0}x_{e_1}$, and $\sigma_2= 4x_{e_0}^2+4x_{e_1}^2-2$.
  Hence, for any $\delta$, the inner summation has   degree 2 in $x_{e_0},x_{e_1}$, and thus $h_n$ has   degree $2\cdot 2^{n-2}=2^{n-1}$ in $x_{e_0},x_{e_1}$, as desired.
\end{proof}

We will also need the following property of the roots $r_\delta$ in (\ref{eq:rdelta}).

\begin{lemma}
\label{lem:z}
For $\delta\in\{\pm 1\}^{E_n\setminus \{e_0\}}$ and $\epsilon_0\in \{\pm 1\}$, there exists $\eta_{\delta}\in\{\pm 1\}$ such that 
\[
r_\delta +\epsilon_0\sqrt{r_\delta^2 -1}= z_{0,\epsilon_0\cdot \eta_\delta \cdot  \delta}.
\]
\end{lemma}
\begin{proof}
As $r_\delta={1\over 2}(z_{0,\delta}+z_{0,-\delta})$, we get $r_\delta^2 -1= {1\over 4}(z_{0,\delta}-z_{0,-\delta})^2$. This uses $z_{0,\delta}z_{0,-\delta} = 1$. Hence, for some  $\eta_\delta\in\{\pm 1\}$, 
 $\sqrt{r_\delta^2-1}=\eta_\delta {1\over 2}(z_{0,\delta}-z_{0,-\delta})$. Thus,
$r_\delta +\epsilon_0\sqrt{r_\delta^2-1}
= {1\over 2} z_{0,\delta}(1+\epsilon_0\eta_\delta) +{1\over 2}z_{0,-\delta}(1-\epsilon_0\eta_\delta)
= z_{0,\epsilon_0\cdot \eta_\delta \cdot \delta},$
as desired.
\end{proof}

\subsection{Cycle polynomials and resultants}\label{sec:resultant}

Here, we show yet another (recursive) representation for the cycle polynomial $\Gamma_n$, as the resultant of two `smaller' cycle polynomials $\Gamma_p$ and $\Gamma_q$, where  $n = p + q - 2$ and $p,q\ge 3$. 

\begin{figure}
\centering
\includegraphics[height = 4cm]{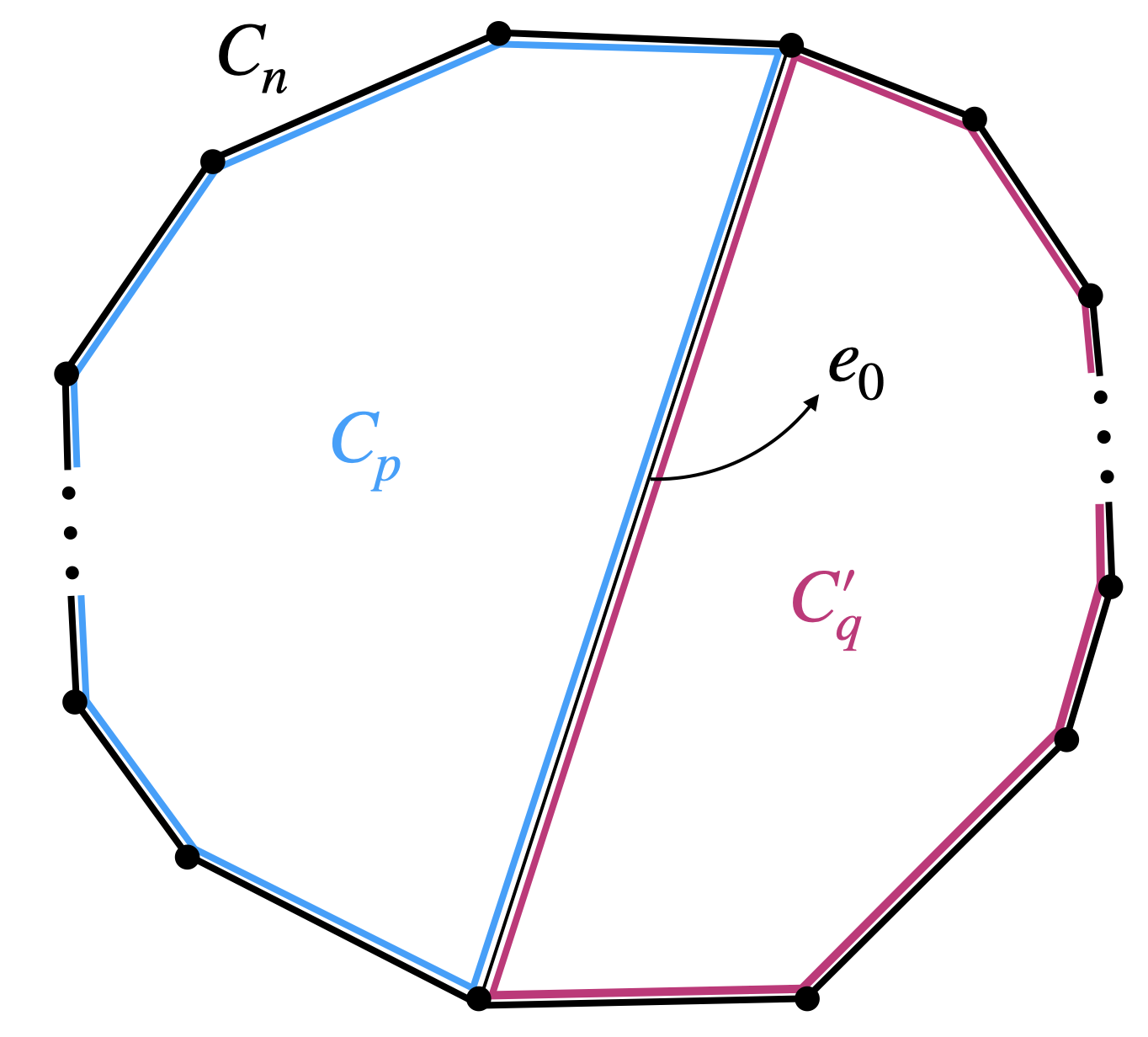}
\caption{The cycle $C_n$ is obtained by deleting the common edge $e_0$ from $C_p$ and $C_q'$.}
\label{fig:pandq}
\end{figure}

The construction goes as follows. Consider the $n$-th cycle $C_n=(1,\ldots,n)$ for $n\ge 4$.
For $3\le p\le n-1$, consider the graph $C_n+e_0$ obtained by adding the edge $e_0=\{1,p\}$ to $C_n$, so that  $e_0$ is a chord of $C_n$ in this graph. Then, consider the cycles $C_p=(1,2,\ldots,p)$ and 
$C'_q=(p, p+1,\ldots, n,1)$ with respective lengths $p$ and $q=n+2-p$.  The graphical picture to keep in mind is shown in Figure \ref{fig:pandq}: by adding a chord $e_0$ to the cycle $C_n$, one gets two smaller cycles $C_p$ and $C_q'$ that share this chord as a unique common edge.
In Figure~\ref{fig:cycle+cord} below, we also see an illustration for $C_4$, where $C_4+e_0 $ is the union of two 3-cycles, and for $C_5$, where $C_5+e_0$ is the union of a 3-cycle and a 4-cycle. 
For convenience, set $F_p=\{\{1,2\},\ldots,\{p-1,p\}\}$, so that the edge set of $C_p$ is $E(C_p)=F_p\cup\{e_0\}$, and set
$F'_q=\{\{p,p+1\}, \ldots, \{n-1,n\},\{n,1\}\}$, so that the edge set of $C_q'$ is $E(C_q')=F'_q\cup \{e_0\}$. 
Then, the cycle polynomials $\Gamma_n=\Gamma_{C_n},\Gamma_{C_p},\Gamma_{C'_q}$ belong to $\R[x_e: e\in E_n\cup\{e_0\}]$, depending on the variables $\{x_e : e\in E_n\}$, $\{x_e: e\in F_p\cup \{e_0\}\}$, and $\{x_e: e\in F_q'\cup\{e_o\}\}$, respectively. We similarly define the polynomials $h_n = h_{C_n}$, $h_{C_p}$ and~$h_{C'_q}$.

The construction of the cycle polynomial $\Gamma_n$ from $\Gamma_{C_p}$ and $\Gamma_{C'_q}$ mimics the geometric observation that the elliptope $\ME(C_n)$ can be obtained from the elliptopes $\ME(C_p)$ and $\ME(C'_q)$ by projecting out the variable corresponding to the chord $e_0$.
On the algebraic side, we will  show that $\Gamma_n$ is  obtained by eliminating $x_{e_0}$ from the equations $\Gamma_{C_p} = \Gamma_{C_q'} = 0$. We make this precise using \emph{resultants}. The theory we need and more can be  found in \cite[Chapter 3, \S 1]{CoxLittleOShea2005UsingAG}. We introduce our notation and recall some facts.

\medskip
For two univariate polynomials  $f = a_0 + a_1 \, y + \cdots + a_{d_f} \, y^{d_f}$ and $g = b_0 + b_1 \, y + \cdots + b_{d_g} \, y^{d_g}$ in $K[y]$, with coefficients $a_k, b_\ell$ in a field $K$ and respective degrees $d_f$ and $d_g$, their  \emph{Sylvester matrix} ${\rm Syl}_y(f,g)$ is the $(d_f+d_g) \times (d_f + d_g)$~matrix 
\[
\mathrm{Syl}_y(f,g)= \begin{pmatrix}
S_f \\ S_g
\end{pmatrix}, \quad \text{where} \quad S_f \, = \, 
\begin{pmatrix}
a_0 & a_1 & \cdots & a_{d_f} &        &        \\
    & a_0 & a_1    & \cdots & a_{d_f} &        \\
    &     & \ddots &        &         & \ddots \\
    &     &        & a_0    & a_1     & \cdots & a_{d_f} 
    \end{pmatrix} \, \in \, K^{d_g \times (d_f+d_g)}
\]
and $S_g$ is constructed analogously. The  resultant ${\rm Res}_y(f,g)$ is defined as the determinant of the Sylvester matrix:
\begin{align}\label{eq:resultant}
{\rm Res}_y(f,g) =\det {\rm Syl}_y(f,g).
\end{align}
 In particular, ${\rm Res}_y(f,g)$ is a polynomial in the coefficients $a_i, b_j$ with integer coefficients.
 An alternative formula for the resultant, 
called \emph{Poisson's formula}, will be useful. 
Let $r_1, \ldots, r_{d_f}\in \bar{K}$ be the roots of $f \in K[y]$ (listed with the appropriate multiplicity) in the algebraic closure  $\bar{K}$   of $K$.
The  resultant ${\rm Res}_y(f,g)$ in (\ref{eq:resultant}) can be equivalently expressed as 
\begin{equation} \label{eq:Poisson}
{\rm Res}_y(f,g) \, = \, a_{d_f}^{d_g} \, \prod_{i = 1}^{d_f} g(r_i).
\end{equation}
This highlights a crucial property of the resultant: it vanishes precisely when $f$ and $g$ have a common root. 
 Using (\ref{eq:resultant}) and (\ref{eq:Poisson}), one can easily check the following properties of the resultant: for $f_1,f_2,f,g\in K[y]$ and $\lambda\in K$,
 \begin{align}
{\rm Res}_y(f_1f_2,g) = {\rm Res}_y(f_1,g){\rm Res}_y(f_2,g), \label{eq:Poisson1}\\
{\rm Res}_y(g,f) = (-1)^{d_fd_g} {\rm Res}_y(f,g), \label{eq:Poisson2} \\
{\rm Res}_y(\lambda \, f, g) = \lambda^{d_g} {\rm Res}_y(f,g).\label{eq:Poisson3}
 \end{align}
Here, we will use the resultant  for $(f,g) = (\Gamma_{C_p}, \Gamma_{C'_q})$ and  $(h_{C_p}, h_{C_q'})$, $y = x_{e_0}$, 
and $K$ is the field of rational functions 
in the variables $(x_e \, : \, e \in E_n)$. Our main result is: 

\begin{theorem} \label{thm:resultant}
${\rm Res}_{x_{e_0}}(\Gamma_{C_p}, \Gamma_{C_q'}) =  \Gamma_n$ for any $n=p+q-2$ with $p,q\ge 3$.
\end{theorem}

Our strategy to prove Theorem \ref{thm:resultant} is to first compute the resultant of the polynomials $h_{C_p}$ and $h_{C_q'}$ and then to apply properties (\ref{eq:Poisson1})-(\ref{eq:Poisson3}) to obtain the resultant of $\Gamma_{C_p}$ and $\Gamma_{C_q'}$. First, we show that the resultant of $h_{C_p}$ and $h_{C_q'}$ is a scaling of $h_n^2$. 

\begin{proposition} \label{prop:Resh}
With the above notation, consider the polynomials $h_n \in \mathbb{C}[x_e  :  e \in E_n]$, $h_{C_p} \in \mathbb{C}[x_e : e \in F_p\cup\{e_0\}]$, and $h_{C'_q} \in \mathbb{C}[x_e  :  e \in F_q'\cup\{e_0\}]$,  where $E_n=F_p\cup F'_q$. 
We have the identity: ${\rm Res}_{x_{e_0}}(h_{C_p}, h_{C'_q}) = (2^{2^{n-1}} \, h_n)^2$. 
\end{proposition}

\begin{proof}
We compute the resultant using relation (\ref{eq:Poisson}) applied to $f= h_{C_p}$ and $g=h_{C'_q}$, viewed as univariate polynomials in $x_{e_0}$.
By Proposition \ref{prop:univleadingterm}, the roots of $h_{C_p}$ 
are given by 
$r_\delta={1\over 2}(z_{0,\delta,\azero}+z_{0,-\delta,\azero})$ for $ \delta \in \{-1,1\}^{F_p}.$
Here, the subscript of $z$ uses the   ordering of $\{e_0\} \cup E_n$ as $e_0, F_p, F_q'$. 
By Poisson's formula \eqref{eq:Poisson}, the resultant is given by the following~expression:
\[
{\rm Res}_{x_{e_0}}(h_{C_p}, h_{C_q'})= (2^{2^{p-1}})^{2^{q-1}} \, \prod_{\delta \in \{-1,1\}^{F_p}}  h_{C'_q}(r_\delta)=
2^{2^n} \prod_{\delta \in \{-1,1\}^{F_p}}  h_{C'_q}(r_\delta),  \]
where we used $p + q - 2 = n$. Using the definition (\ref{eq:hnz}) for the polynomial $h_{C_q'}$, we get
\begin{align*}
h_{C_q'}= \prod_{\delta'\in\{\pm 1\}^{F_q'}} \prod_{\epsilon_0\in \{\pm 1\}} (z_{\epsilon_0,\azero,\delta'}-1) = 
 \prod_{\delta'\in\{\pm 1\}^{F_q'}} \prod_{\epsilon_0\in \{\pm 1\}} (z_{\epsilon_0,\azero,\azero}\cdot z_{0,\azero,\delta'}-1).
 \end{align*}
By Lemma \ref{lem:z}, evaluating $z_{\epsilon_0,\azero,\azero}$ at $x_{e_0}=r_\delta$ gives:
\[
(z_{\epsilon_0,\azero,\azero})_{|x_{e_0} = r_\delta} \,=\, r_{\delta}+\epsilon_0\sqrt{r_\delta^2-1} \,= \,z_{0,\epsilon_0\cdot \eta_\delta\cdot  \delta,\azero}\ \text{   for some } \eta_\delta\in\{\pm 1\}.
\]
Then, viewing $h_{C_q'}$ as a polynomial in $x_{e_0}$ and evaluating it at $x_{e_0}=r_\delta$, 
   we obtain 
\begin{align*}
\prod_{\delta \in \{-1,1\}^{F_p}}  h_{C'_q}(r_\delta)& =
\prod_{\delta \in \{-1,1\}^{F_p}}  \prod_{\delta'\in\{\pm 1\}^{F_q'}}\prod_{\epsilon_0\in \{\pm 1\}}
(z_{0,\epsilon_0\cdot \eta_\delta\cdot  \delta,\azero}\cdot z_{0,\azero,\delta'}-1)\\
&=
\prod_{\delta \in \{-1,1\}^{F_p}}  \prod_{\delta'\in\{\pm 1\}^{F_q'}}
(z_{0,\delta,\delta'}-1)^2\\
&= \prod_{\epsilon\in \{\pm 1\}^{E_n}} (z_{0,\epsilon}-1)^2\\
&= h_n^2.
\end{align*}
Here,  we use  the identity
$\prod_\delta\prod_{\epsilon_0}(z_{0,\epsilon_0\cdot\eta_\delta \cdot \delta,\azero}\cdot z_{0,\azero,\delta'}-1)=\prod_\delta 
(z_{0,\delta,\azero}\cdot z_{0,\azero,\delta'}-1)^2$ for the second equality, combined with $z_{0,\delta,\azero}\cdot z_{0,\azero,\delta'}=z_{0,\delta,\delta'}$. 
The last equality follows using relation (\ref{eq:hn}) and the fact that $E_n=F_p\cup F_q'$.
So, we have proved that ${\rm Res}_{x_{e_0}}(h_{C_p}, h_{C_q'})=2^{2^{n}} h_n^2$, as desired.
\end{proof}

\begin{proof}[Proof of Theorem \ref{thm:resultant}]
Using  Definition \ref{def:cyclepolynomial} and  (\ref{eq:Poisson1})-(\ref{eq:Poisson3}), we have the chain of equalities: 
\[ {\rm Res}(h_{C_p}, h_{C_q'}) \, = \, {\rm Res}(2^{2^{p-1}} \Gamma_{C_p}^2, 2^{2^{q-1}} \Gamma_{C_q'}^2) \, = \, 2^{2^{n+1}} {\rm Res}(\Gamma_{C_p}^2, \Gamma_{C_q'}^2) \, = \, 2^{2^{n+1}} {\rm Res}(\Gamma_{C_p}, \Gamma_{C_q'})^4.
\]
where the resultant is always taken with respect to the common edge variable $x_{e_0}$. 
By Proposition \ref{prop:Resh} (combined with Definition \ref{def:cyclepolynomial}), 
 ${\rm Res}(h_{C_p}, h_{C_q'})=2^{2^{n}}h_n^2=2^{2^{n+1}} \Gamma_n^4$. This implies 
 ${\rm Res}(\Gamma_{C_p}, \Gamma_{C_q'})^4 = \Gamma_n^4$, and thus ${\rm Res}(\Gamma_{C_p}, \Gamma_{C_q'})= \alpha \Gamma_n$ for some $\alpha=\pm 1$. 
 Now, we argue that $\alpha=1$, i.e., $\Gamma_n= R_{p,q}$, setting $R_{p,q}={\rm Res}(\Gamma_{C_p}, \Gamma_{C_q'})$ for compact notation.
  
Let us write $\Gamma_{C_p}= a_0+ a_1x_{e_0}+\ldots+ a_{d_p}x_{e_0}^{d_p}$ and  $\Gamma_{C'_q}= b_0+b_1+\ldots+b_{d_q}x_{e_0}^{d_q}$, where $d_p=2^{p-2}$, $d_q=2^{q-2}$, the $a_i$'s involve the variables $x_e$ for $e\in F_p$, and the $b_i$'s involve the variables $x_e$ for $e\in F'_q$.
 Consider a variable $x_{e_1}$ with $e_1\in F_p$, so that $x_{e_1}$ occurs only in $\Gamma_{C_p}$ and not in $\Gamma_{C'_q}$.  
 By relation (\ref{eq:Gammauniv}), $\Gamma_n$ has leading coefficient 1 when viewed as univariate polynomial in  any given variable, and the same holds for $\Gamma_{C_p}$ and $\Gamma_{C'_q}$. Considering the variable $x_{e_0}$,  this implies  $a_{d_p}=b_{d_q}=1$. Moreover, by considering the variable $x_{e_1}$, this implies that the leading term in $x_{e_1}$ within $\Gamma_{C_p}$ occurs in the coefficient $a_0$ (the only coefficient without a factor $x_{e_0}$).
We now compute the leading coefficient of the resultant $R_{p,q}$, viewed as a polynomial in $x_{e_1}$. For this we use the formula for $R_{p,q}$ given by the determinant of the Sylvester matrix $\text{\rm Syl}_{x_{e_0}}(\Gamma_{C_p},\Gamma_{C'_q})$ as in (\ref{eq:resultant}). The term in the expansion of this determinant which has the highest degree in $x_{e_1}$ is obtained by picking all $d_q$ occurrences of $a_0$ on the nonzero lower diagonal of the submatrix $S_{\Gamma_{C_p}}$ and all $d_p$ occurrences of $b_{d_q}$ on the nonzero upper diagonal of $S_{\Gamma_{C'_q}}$. 
In other words, $x_{e_1}$ appears with highest degree in the term 
$(a_0)^{d_p}\cdot (b_{d_q})^{d_p}=(a_0)^{d_p}$, and thus its leading coefficient is positive. 
On the other hand, as $R_{p,q}=\alpha \Gamma_n$, the leading coefficient in $x_{e_1}$ is equal to $\alpha$. This implies $\alpha=1$, as desired.
   \end{proof}

\begin{example}\label{ex:Gamma4}
For $p = q = 3$,  $n = 4$,
the graph $C_4+e_0$ is obtained by gluing two $3$-cycles along the edge $e_0 = \{1,3\}$, see Figure \ref{fig:cycle+cord} (left). 
\begin{figure}[ht]
    \centering
    \noindent\makebox[\textwidth]{
    \begin{tikzpicture}[scale=1]

\begin{scope}
    \node[fill=black,circle,inner sep=2pt,label=below:1] (a1) at (0,0) {};
    \node[fill=black,circle,inner sep=2pt,label=below:2] (a2) at (2,0) {};
    \node[fill=black,circle,inner sep=2pt,label=above:3] (a3) at (2,2) {};
    \node[fill=black,circle,inner sep=2pt,label=above:4] (a4) at (0,2) {};

    \draw (a1)--(a2)--(a3)--(a4)--(a1);

    \draw (a1)--(a3);
\end{scope}

\begin{scope}[xshift=5cm]
    \node[fill=black,circle,inner sep=2pt,label=above:4] (b4) at (0,2) {};
    \node[fill=black,circle,inner sep=2pt,label=right:3] (b3) at (1.14,1.24) {};
    \node[fill=black,circle,inner sep=2pt,label=right:2] (b2) at (0.70,0) {};
    \node[fill=black,circle,inner sep=2pt,label=left:1] (b1) at (-0.70,0) {};
    \node[fill=black,circle,inner sep=2pt,label=left:5] (b5) at (-1.14,1.24) {};

    \draw (b1)--(b2)--(b3)--(b4)--(b5)--(b1);
    \draw (b1)--(b3);
\end{scope}
\end{tikzpicture}
    }
    \caption{With $e_0=\{1,3\}$, the graph $C_4+e_0$ is the union of $C_3=(1,2,3)$ and $C_3'=(3,4,1)$, while $C_5+e_0$ is the union of $C_3=(1,2,3)$ and $C_4'=(3,4,5,1)$.}
    \label{fig:cycle+cord}
    \end{figure}
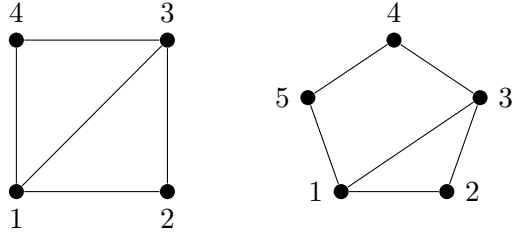
    Theorem \ref{thm:resultant} gives the following way to 
    obtain  $\Gamma_4$ as a $4 \times 4$ Sylvester determinant, using $\Gamma_3$ (which was computed in Example \ref{ex:Gamma3}):
    \[ \Gamma_4 \, = \, \det \begin{pmatrix}
    -1 + x_{12}^2 + x_{23}^2 & -2 x_{12}x_{23} & 1 & 0 \\ 
    0 & -1 + x_{12}^2 + x_{23}^2 & -2 x_{12}x_{23} & 1 \\ 
    -1 + x_{14}^2 + x_{34}^2 & -2 x_{14}x_{34} & 1 & 0 \\ 
    0 &  -1 + x_{14}^2 + x_{34}^2 & -2 x_{14}x_{34} & 1
    \end{pmatrix} .\]
    Similarly, for $p=3$, $q=4$, $n=5$, $C_5+e_0$ is obtained by gluing a 3-cycle and a 4-cycle along the edge $e_0=\{1,3\}$, as shown in Figure \ref{fig:cycle+cord} (right). By Theorem \ref{thm:resultant},     the polynomial $\Gamma_5$ can be obtained from $\Gamma_3$ and $\Gamma_4$  as a $6 \times 6$ Sylvester determinant.
\end{example}

\begin{remark}\label{rem:newproof}
Observe that the result in Theorem \ref{thm:resultant} offers a new proof of the fact that $\Gamma_n$ is a multivariate polynomial, as we already established in Theorem \ref{lem:existence}.
Indeed, its proof relies only on the fact that $\Gamma_{C_p}$ and $\Gamma_{C_q'}$ are {\em univariate polynomials} in the variable $x_{e_0}$. 
Then, using the definition (\ref{eq:resultant})  of the resultant via the Sylvester matrix, the identity $\Gamma_n=\text{\rm Res}_{x_{e_0}}(\Gamma_{C_p},\Gamma_{C'_q})$
shows by induction that $\Gamma_n$ is a {\em multivariate polynomial}. The base case is for $\Gamma_3$, which is indeed a multivariate polynomial as we saw in Example \ref{ex:Gamma3}. In addition, the same argument shows that $\Gamma_n$ is a multivariate polynomial with {\em integer coefficients} for all $n \geq 3$.
\end{remark}

\begin{remark} \label{rem:newproofnew}
Using the definition (\ref{eq:resultant})  of the resultant via the Sylvester matrix, the identity $\Gamma_n=\text{\rm Res}_{x_{e_0}}(\Gamma_{C_p},\Gamma_{C'_q})$
shows by induction that $\Gamma_n$ is a multivariate polynomial with \emph{integer coefficients}. The base case is for $\Gamma_3$, which satisfies this property by Example \ref{ex:Gamma3}. Moreover, one can use $\Gamma_n=\text{\rm Res}_{x_{e_0}}(\Gamma_{C_p},\Gamma_{C'_q})$ as an alternative \emph{definition} of $\Gamma_n$, avoiding the use of square~roots as is now done in Definition \ref{def:cyclepolynomial}. 
\end{remark}

Via Theorem \ref{thm:resultant}, Sylvester's formula gives several determinantal expressions for $\Gamma_n$, one for each pair $(p,q)$ such that $p+q=n+2$ and $p,q\ge 3$, in which case 
$\Gamma_n$ is the determinant of a Sylvester matrix of size
$2^{p-2}+2^{q-2}$.  
One can  check that the smallest size is equal to $\sqrt {2^n}$, obtained with $p=q={n\over 2}+1$ if $n$ is even, and equal to ${3\over 2}\sqrt {2^{n-1}}$, obtained with $p={n+1\over 2}$, $q={n+3\over 2}$  if $n$ is odd.
In comparison, the determinantal represention of $\Gamma_n$ obtained from the multiplication matrix $M_{g_n}$  in (\ref{eq:hn}), involves a matrix of size $2^n$, and its determinant is $h_n = 2^{2^{n-1}} \Gamma_n^2$. On the other hand, that $2^n \times 2^n$-matrix has significantly simpler entries. The preferred representation may depend on the precise application. Compact representations and efficient evaluation of $\Gamma_n$ form an interesting topic for further research. 

\subsection{Homogeneous cycle polynomials}\label{sec:cyclehomogeneous}

Recall that the elliptope ${\cal E}(G)$ consists of partially filled correlation matrices which allow a positive semidefinite completion. Replacing \emph{correlation matrices} by \emph{semidefinite matrices} in this definition gives the cone ${\MCscr}(G)$. 
We think of a partially filled  matrix as a point  $s\in \mathbb{R}^{E\cup V}=\R^E\times \R^V$, 
where   the coordinates on the first factor are $s_{ij}$, indexed by the edges $\{i,j\}\in E$, and 
the coordinates on the second factor are $s_{kk}$ for $k\in V$. 
Define
\begin{align}\label{eq:coneCG}
 \MCscr(G) \, = \, \{ s \in \mathbb{R}^{E\cup V}  :  \exists X \in {\cal S}^n_+ \text{ such that } 
X_{ij}=s_{ij} \text{ for } \{i,j\}\in E,\ X_{kk}=s_{kk}\text{ for } k\in V\}.\end{align}
A point $s\in\R^{E\cup V}$ lies in $\MCscr (G)$ if and only if the point $x\in \R^E$ with entries 
\begin{equation} \label{eq:substitution} 
x_{ij} \, = \, \frac{s_{ij}}{\sqrt{s_{ii}s_{jj}}}  \ \text{ for } \{i,j\}\in E
\end{equation}
lies in ${\cal E}(G)$  (setting $x_{ij}=0$ if $s_{ii}s_{jj}=0$). For this reason, the authors of \cite{Sturmfels2010MultivariateGaussians} are interested in a homogeneous version of the cycle polynomial expressed in the variables $s_{ij}, s_{kk}$. In this section, we prove Conjecture~4.9 in \cite{Sturmfels2010MultivariateGaussians}, which pertains to the degree of that homogeneous polynomial. To define it, we substitute \eqref{eq:substitution} in $\Gamma_n$. The resulting function of the $s$-variables is denoted by $R_n$. 

\begin{proposition} \label{prop:Rrational}
    The function $R_n$, obtained by substituting $x_{ij} = \tfrac{s_{ij}}{\sqrt{s_{ii}s_{jj}}}$ in the cycle polynomial $\Gamma_n$, is a rational function in the $s$-variables. Its denominator is $s_{11}^{2^{n-3}} \cdots s_{nn}^{2^{n-3}}$.
\end{proposition}

\begin{proof}
    Notice that $R_n$ is a fraction whose numerator is a polynomial in $s$, and whose denominator is a product of powers of $s_{11}, \ldots, s_{nn}$. A priori, these are fractional powers. We must show that each of the variables $s_{11}, \ldots, s_{nn}$ appears in the denominator with an integer exponent. Fix $i \in [n]$ and let $x_{i-1,i}$ and $ x_{i,i+1}$ be the edge variables of the two edges containing $i$ (taking indices modulo $n$). We set $E' = \{\{i-1,i\},\{i,i+1\}\}$ and we expand $\Gamma_n$ as in relation (\ref{eq:cycle-even}):
    \[ \Gamma_n \, = \, \sum_{(\alpha,\beta) \in \mathbb{N}^2} c_{\alpha,\beta}(x_{E \setminus E'}) \, x_{i-1,i}^\alpha \, x_{i,i+1}^\beta .\]
    The substitution $x_{ij} = \frac{s_{ij}}{\sqrt{s_{ii}s_{jj}}}$ leads to the expression 
    \[ \sum_{(\alpha,\beta) \in \mathbb{N}^2} \tilde{c}_{\alpha,\beta} \, 
    \frac{s_{i-1,i}^\alpha \, s_{i,i+1}^\beta}
    { (s_{i-1,i-1})^{\alpha\over 2} (s_{i+1,i+1})^{\beta\over 2} (s_{ii})^{\alpha+\beta\over 2} }
    \]
    where $\tilde{c}_{\alpha, \beta}$ is obtained from $c_{\alpha,\beta}$ after the substitution. Note that $\tilde{c}_{\alpha, \beta}$ does not involve the variable $s_{ii}$. 
    Moreover, by Corollary \ref{cor:even}, each $\alpha + \beta$ for which $\tilde{c}_{\alpha, \beta} \neq 0$ is an even integer. Therefore, $s_{ii}$ appears with an integer exponent in each term for all $i$. By Proposition \ref{prop:bivariatedegree}, $\Gamma_n$ has degree $2^{n-2}$ when viewed as a bivariate polynomial in $x_{i-1,i}$ and $ x_{i,i+1}$. Hence,  the exponent of $s_{ii}$ in the common denominator is $2^{n-3}$. 
\end{proof}

\begin{example}
For $n = 4$, the rational function $R_4$ is given by 
\[\begin{matrix} s_{12}^2 \, s_{23}^2 \, s_{34}^2 \, s_{14}^2 \cdot R_4 \, = \, s_{14}^4 s_{22}^2 s_{33}^2 - 4 s_{12} s_{14}^3 s_{22} s_{23} s_{33} s_{34} + 
 4 s_{11} s_{14}^2 s_{22} s_{23}^2 s_{34}^2 + 4 s_{12}^2 s_{14}^2 s_{22} s_{33} s_{34}^2 \\- 
 2 s_{11} s_{14}^2 s_{22}^2 s_{33} s_{34}^2 - 4 s_{11} s_{12} s_{14} s_{22} s_{23} s_{34}^3 + 
 s_{11}^2 s_{22}^2 s_{34}^4 + 4 s_{12}^2 s_{14}^2 s_{23}^2 s_{33} s_{44} \\ - 
 2 s_{11} s_{14}^2 s_{22} s_{23}^2 s_{33} s_{44} - 2 s_{12}^2 s_{14}^2 s_{22} s_{33}^2 s_{44} - 
 4 s_{11} s_{12} s_{14} s_{23}^3 s_{34} s_{44} - 4 s_{12}^3 s_{14} s_{23} s_{33} s_{34} s_{44} \\+ 
 8 s_{11} s_{12} s_{14} s_{22} s_{23} s_{33} s_{34} s_{44} + 4 s_{11} s_{12}^2 s_{23}^2 s_{34}^2 s_{44} - 
 2 s_{11}^2 s_{22} s_{23}^2 s_{34}^2 s_{44} - 2 s_{11} s_{12}^2 s_{22} s_{33} s_{34}^2 s_{44} \\ + 
 s_{11}^2 s_{23}^4 s_{44}^2 - 2 s_{11} s_{12}^2 s_{23}^2 s_{33} s_{44}^2 + s_{12}^4 s_{33}^2 s_{44}^2. \end{matrix} \]
 Each term on the righthand side corresponds to a term of $\Gamma_4$, see Example \ref{ex:smallGamma}.
\end{example}

\begin{definition}\label{def:hGamma}
    The \emph{homogeneous cycle polynomial} is $\Gamma_n^h = s_{11}^{2^{n-3}} \cdots s_{nn}^{2^{n-3}} \cdot R_n$. 
\end{definition}

We note that our notation differs slightly from \cite{Sturmfels2010MultivariateGaussians}, where the cycle polynomial $\Gamma_n$ and the homogeneous cycle polynomial $\Gamma_n^h$ are denoted by $\Gamma_n'$ and $\Gamma_n$ respectively. 
The following theorem settles Conjecture 4.9 in \cite{Sturmfels2010MultivariateGaussians}. 

\begin{theorem} \label{thm:degGammah}
  The polynomial $\Gamma_n^h$  is homogeneous and its degree is  $n \cdot 2^{n-3}$.
\end{theorem}

\begin{proof}
    By construction, the rational function $R_n$ is homogeneous of degree zero. Multiplying with the denominator $s_{11}^{2^{n-3}} \cdots s_{nn}^{2^{n-3}}$, which is homogeneous of degree $n\cdot 2^{n-3}$,  gives a homogeneous polynomial of degree $n \cdot 2^{n-3}$. 
\end{proof}

In fact, the homogeneous cycle polynomial is homogeneous with respect to a \emph{multigrading} by the group $\mathbb{Z}^n$. We define the multidegree of $s_{ij}$ as $\deg_n(s_{ij}) = {\bf e}_i + {\bf e}_j$. Here, ${\bf e}_j$ is the $j$-th standard basis vector of $\mathbb{Z}^n$. Note that $R_n$ is homogeneous of multidegree ${\bf 0} \in \mathbb{Z}^n$.

\begin{theorem} \label{thm:multidegGammah}
    The homogeneous cycle polynomial $\Gamma^h_n$ is homogeneous with respect to the multigrading $\deg_n$. We have $\deg_n(\Gamma_n^h) = (2^{n-2}, \ldots, 2^{n-2})$. 
\end{theorem}

Theorem \ref{thm:multidegGammah} can be equivalently expressed as follows. For any $\lambda \in (\mathbb{C}^*)^n$, we~have 
\[ \Gamma_n^h(\lambda_1\lambda_2 \, s_{12}, \ldots, \lambda_1 \lambda_n \, s_{1n}, \lambda_1^2 \, s_{11}, \ldots, \lambda_n^2 s_{nn}) \, = \, \lambda_1^{2^{n-2}} \cdots \lambda_n^{2^{n-2}} \cdot \Gamma_n^h(s_{12}, \ldots, s_{1n}, s_{11}, \ldots, s_{nn}).\]

\section{The algebraic boundary of the elliptope of a graph}\label{sec:boundary}

Recall that the \emph{Zariski closure} $\overline {S}^z$ of a set $S\subseteq \R^n$ is the smallest (with respect to inclusion) affine variety in $\C^n$ that contains $S$, and the {\em algebraic boundary} $\partial_a S=\overline{\partial S}^z \subseteq \mathbb{C}^n$ is the Zariski closure of its euclidean boundary $\partial S$.

In this section we investigate the algebraic boundary of the elliptope $\ME(G)$ and of its outer approximation $\tME(G)$ for a graph $G$. Recall the definition of the semialgebraic sets $S(G)$ and $\tME(G)=S(G)\cap\cos(\pi\MET(G))$  in Equations (\ref{eq:SG}) and  (\ref{eq:tMEG}).
Our main result is that  
 the algebraic boundary of  $\tME(G)$  is given by the cycle polynomials $\Gamma_C$ for the chordless cycles $C$ of $G$ and the determinantal polynomials $\det(x[K])$ for the maximal cliques $K$ (see Theorem~\ref{thm:boundary-cycle-completable}). This directly provides the algebraic boundary of the elliptope $\ME(G)$ when $G$ is cycle completable, in which case $\ME(G) = \tME(G)$.
We also characterize the graphs whose elliptope $\ME(G)$ is a spectrahedron (Corollary~\ref{cor:spectrahedron}). In addition, we extend the results to the cone $\MCscr(G)$ from (\ref{eq:coneCG}) (Corollary~\ref{cor:coneCG}) and to the graphs obtained by adding an apex node to a cycle completable graph (Proposition~\ref{prop:suspensionboundary}).

\subsection{Some geometric properties of graph elliptopes}

We begin with  some preliminary geometric results on the elliptope of a graph.

\begin{proposition}\label{proposition:piboundEn}
    Let $G=(V=[n],E)$ be a graph. 
    We have ${\rm int} ({\cal E}(G))=\pi_E({\rm int} ( {\cal E}_n))$ and 
 $\partial {\cal E}(G)\subseteq \pi_E(\partial {\cal E}_n)$.
\end{proposition}

\begin{proof}
 Let $x_E\in \pi_E( {\rm int}({\cal E}_n))$, i.e., there exists $X\in {\rm int}({\cal E}_n)$ such that $x_E=\pi_E(X)$. Since $ X\in {\rm int}({\cal E}_n)$, there exists a ball $B_{\epsilon}(X)\subset {\cal E}_n$ for some small $\epsilon>0$. 
    This means that small perturbations of the entries of $X$ result in a positive definite matrix. This clearly holds if we restrict only to the variables labeled with $E$. 
    This implies that there exists a small ball around $x_E$ consisting of points $x\in\R^E$  that can be completed to positive semidefinite matrices, thus showing $x_E\in {\rm int} ({\cal E}(G))$.
    This shows the inclusion $\pi_E( {\rm int}({\cal E}_n))\subseteq {\rm int} ({\cal E}(G))$. 
    Then, the following chain of inclusions holds:
    \[
    \partial{\cal E}(G) = {\cal E}(G)\setminus {\rm int}({\cal E}(G))\subseteq \pi_E({\cal E}_n)\setminus \pi_E( {\rm int}({\cal E}_n)) \subseteq \pi_E({\cal E}_n\setminus {\rm int}({\cal E}_n)) = \pi_E(\partial {\cal E}_n).
    \]
   We now show the reverse   inclusion ${\rm int} ({\cal E}(G)) \subseteq \pi_E( {\rm int}({\cal E}_n))$.
   Equivalently, we show that if $x_E\in {\cal E}(G)$ with $x_E \not\in \pi_E( {\rm int}({\cal E}_n))$, then $x_E\not\in \text{ int}(\ME(G))$.
    By assumption, the system 
    \[
    \label{eq:polysys}
    \left\{
    \begin{aligned}
    & X \succeq 0, \\
    & \langle E_{ii},X\rangle =1,  \text{ for } i\in V,\\
    & \langle E_{ij},X\rangle =x_{ij}, \text{ for }\{i,j\} \in E,
    \end{aligned}
    \right.
    \]
    admits a solution $X$, but it does not admit a positive definite solution. Here, the matrices $E_{ij}=(e_ie_j^T+e_je_i^T)/2$ are the elementary matrices for $1\le i\le j\le n$.
    By the Theorem of Alternatives for  semidefinite programming  (see, e.g., \cite[Lemma 3.6]{LV-MP-2014}), there exist scalars $ y_i$ ($i\in V$) and $y_{ij}$ ($\{i,j\}\in E$) satisfying the system
    \begin{equation}\label{eq:SDPalternatives}
    \left\{
    \begin{aligned}
    &  \Omega = \sum_{i\in V} y_i E_{ii}+\sum_{\{i,j\}\in E} y_{ij}E_{ij}\succeq 0, \ \Omega \ne 0,\\
    &\sum_{i\in V} y_i + \sum_{\{i,j\}\in E} y_{ij}x_{ij} \leq 0.
    \end{aligned}
    \right.
    \end{equation}
Observe that   $\sum_{i\in V} y_i + \sum_{\{i,j\}\in E} y_{ij} z_{ij} \geq 0$ holds for any $z\in {\cal E}(G)$. Indeed, if $z\in \ME(G)$, there exists $Z\in {\cal E}_n$ such that $z=\pi_E(Z)$. Since $Z,\Omega \succeq 0$, we have  $\langle Z,\Omega\rangle \geq 0$, and expanding the scalar product proves the claim. 
 Moreover,   by \Cref{eq:SDPalternatives}, we derive that $x_E$ satisfies the above inequality at equality, 
  which implies that $x_E \in \partial{\cal E}(G)$, and thus  $x_E \not\in {\rm int}({\cal E}(G)$). So, we have shown equality ${\rm int} ({\cal E}(G))=\pi_E({\rm int} ( {\cal E}_n))$.
\end{proof}

\begin{proposition}\label{proposition:cosboundMET}
For any graph $G=(V,E)$, the following assertions hold.
\begin{itemize}
\item[(i)] $\text{\rm int}(\cos(\pi \MET(G)))= \cos(\pi \text{\rm int}(\MET(G)))$ and 
$\partial (\cos(\pi \MET(G)))= \cos(\pi \partial \MET(G))$.
\item[(ii)]
$\text{\rm int}(S(G))=\{x\in \R^E: x[K]\succ 0 \text{ for all } K\in \MK_G\}$. 
\item[(iii)] $\text{\rm int}(\tME(G)) = \text{\rm int}(S(G)) \cap \cos(\pi \text{\rm int}(\MET(G)))$.
\end{itemize}
\end{proposition}

\begin{proof}
For (i), it suffices to show the claim about the interiors since the claim about the euclidean closures then follows from the fact that the map $\cos:[0,\pi]^n \rightarrow [-1,1]^n$ is one-to-one;
 the claim about the interiors follows easily using Lipschitz continuity of the cosine and arccosine maps.
The proof of (ii) is easy, and (iii) directly follows from (i),(ii), using the fact that the interior of an intersection equals the intersection of the interiors.
\end{proof}

\begin{proposition}\label{proposition:VPn}
For a graph $G=(V,E)$, define the polynomial 
\begin{align}\label{eq:boundaryG}
P_G(x)= \prod _{K\in\MK_G} \det x[K] \cdot \prod_{C\in \MC_G} \Gamma_C(x),
\end{align}
where $\MK_G$ is the set of maximal cliques and  $\MC_G$ is the set of chordless cycles of length at least 4 in $G$.
We have  $\partial \tME(G)\subseteq  \tME(G) \cap V(P_G).$
\end{proposition}

\begin{proof}
By Proposition \ref{proposition:cosboundMET}(iii),    $\partial \tME(G)= \tME(G)  \setminus \text{\rm int}(\tME(G))\subseteq \partial S(G) \cup \partial (\cos(\pi\MET(G)))$.
Let $x=\cos(\pi a)\in \partial \tME(G)$ with $a\in[0,1]^E$. If $x\in\partial S(G)$ then, by Proposition \ref{proposition:cosboundMET}(ii),  we have $\det x[K]=0$ for some   $K\in\MK_G$, and thus $P_G(x)=0$.  
If $x \in \partial \tME(G) \setminus \partial S(G)$, then we have $x\in (-1,1)^E$. Moreover,  $x\in \partial (\cos(\pi \MET(G)))=\cos(\pi \partial \MET(G))$, and thus $a\in\partial \MET(G)$ with $a \in (0,1)^E$ by Proposition \ref{proposition:cosboundMET}(i). Hence, there exists $C\in \MC_G$ for which $a$ satisfies a cycle inequality (\ref{eq:cycle}) at equality. From Remark \ref{rem:cosLn},  we obtain that $\Gamma_C(x)=0$, which implies~$P_G(x)=0$.
  \end{proof}

\begin{remark}\label{rem:strictinclusion}
 For $G=C_3$,  the polynomial $P_{C_3}$ has a single factor corresponding to the (unique) maximal clique $K = [3]$ and  equal to the irreducible cubic $\Gamma_3$. The equality $\partial \ME(C_3)=\ME(C_3)\cap V(P_{C_3})$ holds. For the cycle $G=C_n$ with  $n\ge 4$, we have
\[
    P_{C_n} = (1-x_{12}^2) \cdots (1-x_{n-1,n}^2)(1-x_{1n}^2) \cdot \Gamma_n.
\]
So, $P_{C_n}$ has  $n$ factors corresponding to the maximal cliques (i.e., the edges of $C_n$), and one factor $\Gamma_n$ corresponding to the unique chordless cycle in ${\cal C}_{C_n} = \{C_n\}$.
As we now observe, the inclusion  $\partial \ME(C_n)\subsetneq \ME(C_n)\cap V(P_{C_n})$ is strict for $n\ge 4$. For this, consider the vector 
$x=\cos(\pi a)\in \R^n$, with  $a={n-2\over n}\one $ for $n$  even, and $a={n-3\over n}\one $ for $n$  odd. Then, one can check that the point $a$ belongs to the interior of the metric polytope $\MET(C_n)$ since it satisfies all cycle inequalities (\ref{eq:cycle}) strictly. Hence, by Proposition \ref{proposition:cosboundMET}(i), $x\in {\rm int}(\ME(C_n))$. On the other hand, we have $x\in V(P_{C_n})$. Indeed, $x$ satisfies $\Gamma_n(x)=0$, which follows from 
Remark \ref{rem:cosLn}  
since $a\cdot \one \in 2\Z$. In other words, for the cycle $C_n$ of length $n\ge 4$, we have 
$\text{\rm int}(\ME(C_n)) \cap V(\Gamma_n)\ne \emptyset$.
The fact that $V(\Gamma_4)$ intersects the interior of $\ME(C_4)$ can be seen from the first two plots in the last row of Figure \ref{fig:E4}.
\end{remark}

\subsection{The algebraic boundary of the elliptope for cycle completable graphs}
\label{sec:boundary-cycle-completable}

In this section we describe the algebraic boundary of the elliptope $\ME(G)$ when the graph $G$ is cycle completable (Theorem \ref{thm:boundary-cycle-completable}). In fact, we show a stronger result and characterize the algebraic boundary of the larger set $\tME(G)$ from (\ref{eq:tMEG}) for any graph $G$.

Recall that 
$\MK_G$ is the set of maximal cliques of $G$. To avoid trivialities we assume  $G$ has no isolated nodes, so   each maximal clique has size at least 2.  
Hence, $\MK_G=\MK_{G,2}\cup\MK_{G,\ge 3}$, where $\MK_{G,2}$ denotes the set of maximal cliques of size 2 (i.e., the set of edges $e\in E$ that are not contained in a larger clique), and $\MK_{G,\ge 3}=\{K\in \MK_G: |K|\ge 3\}$.
Let $\MC_G$ denote the set of  chordless cycles of $G$ of length at least 4. Since a clique $K$  of size 3 can also be viewed as a (chordless) cycle and $\det x[K]=-\Gamma_K(x)$, we also define the set 
$\MC_{G,\ge 3}$ as the set of chordless cycles of length at least 3. 
The following is the main result of this section.

\begin{theorem}\label{thm:boundary-cycle-completable}
For any graph $G$, the algebraic boundary of the set $\tME(G)$ is the hypersurface defined by the polynomial $P_G$ from   (\ref{eq:boundaryG}), 
i.e., $\partial_a\tME(G)= V(P_G).$ In particular, if $G$ is cycle completable, then $\partial_a\ME(G)= V(P_G).$
\end{theorem}

We begin with proving  Theorem \ref{thm:boundary-cycle-completable} for graphs with no $K_4$-minors. For this, we  describe  the algebraic boundary of the set $\cos(\pi\MET(G))$ for any graph $G$, in which case the arguments are simpler and exploit the description of the euclidean boundary in Proposition \ref{proposition:cosboundMET}(i).

\begin{proposition}\label{prop:aboundarycosMET}
For any graph $G=(V,E)$, we have $\partial_a \cos(\pi\MET(G)) = V(Q_G)$, where the polynomial $Q_G$ is defined by
\[
Q_G(x)=  \prod_{e\in \MK_{G,2}} (x_e-1)(x_e+1) \cdot \prod_{C\in \MC_{G,\ge 3}}   \Gamma_C(x).
\]
Therefore, if $G$ has no $K_4$-minor, then we have $\partial_a\ME(G)= V(P_G)$. 
\end{proposition}

\begin{proof}
If $G$ has no $K_4$-minor, then $\ME(G)=\cos(\pi\MET(G))$, and $P_G=Q_G$ since $G$ has no clique of size 4. Hence, it suffices to show that $\partial_a\cos(\pi\MET(G))=V(Q_G)$ for any graph~$G$.

 We begin with  the explicit  description of the facets of $\MET(G)$. For each edge $e\in \MK_{G,2}$ and $\alpha\in \{0,1\}$, the set $\MF_{e,\alpha}=\MET(G)\cap\{a\in\R^E: a_e=\alpha\}$ is a facet of $\MET(G)$ and, for each $C\in\MC_{G,\ge 3}$ and $F\subseteq E(C)$ with $|F|$ odd, also the set 
 $\MF_{C,F}=\MET(G)\cap\{a\in \R^{E}: a(F)-a(E(C)\setminus F)=|F|-1\}$ is a facet, and these are all the facets of $\MET(G)$. That is,  
\[\partial \MET(G)=\bigcup_{e\in\MK_{G,2},\alpha \in\{0,1\}}\MF_{e,\alpha} \  \cup
 \bigcup_{C\in \MC_{G,\ge 3},\ F\subseteq E(C),\ |F| \text{\rm odd}} \MF_{C,F}.
 \]
Applying the componentwise cosine map to both sides one obtains that 
\begin{equation} \label{eq:bdrynoK4minor} 
\cos (\pi\partial \MET(G))= \bigcup_{e,\alpha}\cos (\pi \MF_{e,\alpha})\cup\bigcup_{C,F}\cos(\pi \MF_{C,F}).
\end{equation}
It is clear that $\cos (\pi \MF_{e,\alpha}) \subseteq V(x_e-\cos(\pi \alpha))$, and $\cos(\pi \MF_{C,F}) \subseteq V(\Gamma_C)$ holds by Remark~\ref{rem:cosLn}. By taking the Zariski closure on both sides of these two inclusions, we obtain
$\overline{\cos (\pi \MF_{e,\alpha})}^z  \subseteq V(x_e-\cos(\pi \alpha))$ and $ \overline{\cos(\pi \MF_{C,F})}^z \subseteq V(\Gamma_C)$. Here, $\overline{\cdot}^z$ is the Zariski closure.
The key observation is that  
the latter two inclusions are in fact equalities, which we prove below using a dimension argument.
Once we establish this claim,   Proposition~\ref{prop:aboundarycosMET}   follows directly from \eqref{eq:bdrynoK4minor} using  the fact that the closure of the union is the union of closures. 
Let $\MF$ denote one of the facets  $\MF_{e,\alpha}$ or $\MF_{C,F}$. 
Both the relative interior ${\rm relint}( \pi \MF )$ of $\MF$ and its   image ${\rm cos}({\rm relint}(\pi \MF)) $ are real analytic manifolds of dimension $|E|-1$, since the cosine map is one-to-one and thus preserves the dimension.
Therefore, the smallest variety containing ${\rm cos}({\rm relint}(\pi \MF))$ must have dimension at least $|E|-1$. The varieties $V(x_e - \cos(\pi \alpha))$ and $V(\Gamma_C)$ have dimension $|E|-1$ and they are irreducible. Hence, the equalities $\overline{\cos (\pi \MF_{e,\alpha})}^z = V(x_e-\cos(\pi \alpha))$ and $ \overline{\cos(\pi \MF_{C,F})}^z =V(\Gamma_C)$ hold, as desired. The correspondence between facets of $\MET(G)$ and their images under the cosine map is nicely seen, for $G = C_4$, from the color coding in Figure \ref{fig:E4}.
\end{proof}

For the description of the algebraic boundary of the set $\tME(G)$, there are additional determinantal hypersurfaces $\det(K)=0$ (for the cliques of $G$) 
that need to be added  
to the hypersurfaces that form part of the algebraic boundary boundary of $\cos(\pi \MET(G))$. We need some additional technical arguments in order to show Theorem \ref{thm:boundary-cycle-completable}. We state three auxiliary results that will be key ingredients in the proof. The first key ingredient is the fact that each irreducible component of the algebraic boundary of a full-dimensional convex semialgebraic set has codimension 1. 
 
\begin{proposition}\label{lemma:algboundconvex} \cite[Lemma 2.5]{sinn2015algebraicboundary}
    Let $ S \subsetneq \R^n$ be a convex semialgebraic set with non-empty interior. Each irreducible component of the algebraic boundary of $S$ has codimension $1$ in $\C^n$, i.e., $\partial_a S$ is a hypersurface.
\end{proposition}

A second key  ingredient is  that the determinant of a correlation matrix of size at least 3 is an irreducible polynomial. For lack of a precise reference, we include a proof. 

\begin{lemma}\label{lem:det-irreducible}
The determinant $\det (X_n)\in \C[x_{ij}: 1\le i<j\le n]$ of a correlation matrix $X_n \in \ME_n$ of size $n\ge 3$, whose off-diagonal entries are variables $x_{ij}$, is an irreducible polynomial of degree~$n$.
\end{lemma}
\begin{proof}
It is clear that $\det(X_n)$ has degree at most $n$. To see that the degree is equal to $n$, notice that the cyclic term $x_{12}x_{23} \cdots x_{n-1,n} x_{1n}$ appears with coefficient $(-1)^{n-1} 2$. Indeed, this term comes from the two cyclic permutations of $[n] = \{1, \ldots, n \}$.

Our proof of irreducibility is by induction.  The base case is $n = 3$, which is settled since $\det (X_3)=-\Gamma_3$ is indeed irreducible. Assume $n\ge 4$. Laplace expansion of $\det(X_n)$ along the last row (or column) reveals that $\det(X_n) = \det(X_{n-1}) + x_{1n} q_1 + \cdots + x_{n-1,n} q_{n-1}$, where $\det(X_{n-1}) \in \mathbb{C}[x_{ij} : 1 \leq i < j \leq n-1]$ is the upper left principal $(n-1) \times (n-1)$ minor of $X_n$, and $ \pm q_i$ are other $(n-1) \times (n-1)$ minors of $X_n$. Restricting $\det(X_n)$ to the coordinate subspace $L = \{ X \in \MS^n \, : \, x_{jn} = 0 \text{ for } j = 1, \ldots, n-1 \}$, we see that $\det(X_n)_{|L} = \det(X_{n-1})$.

Suppose that $\det(X_n) = f \cdot g$ for some $f, g \in \C[x_{ij}: 1\le i<j\le n]$. In particular, we have $\det(X_n)_{|L} = f_{|L} \cdot g_{|L}$. By the induction hypothesis, $\det(X_{n-1})$ is irreducible. Hence, we may assume without loss of generality that $f_{|L} = c \, \det(X_{n-1})$ and $g_{|L} = c^{-1}$ for some nonzero $c \in \C$. This implies, by the fact that $\deg(\det(X_{n-1})) = n-1$, that $f$ has degree $n-1$ in the variables $x_{ij},  1 \leq i < j \leq n-1$. The polynomial $\det(X_n)$ has at most degree $n-1$ in these variables, so that the degree of $g$ in $x_{ij}, \ 1 \leq i < j \leq n-1$, must be zero. Therefore, $\det(X_n)$ has a factor $g$ which only depends on the variables $x_{jn},$ $ j = 1, \ldots, n-1$. Assume $g$ is not a constant. Then, by symmetry, for each $k\in [n]$, $\det(X_n)$ must also have a non-constant factor which only depends on $x_{jk},$ $ j\in [n] \setminus \{k\}$. This contradicts $\det(X_n)_{|L} = \det(X_{n-1})$, so we conclude that $g$ must be a constant, and $\det(X_n)$ is irreducible.
\end{proof}

Note that the above lemma fails for a correlation matrix $X_n$ of size $n=2$, in which case  $\det (X_2)=1-x_{12}^2=(1-x_{12})(1+x_{12})$. 
The third key ingredient for the proof of Theorem~\ref{thm:boundary-cycle-completable} is the following technical result, showing that each of the varieties $V(\det x[K])$ and $V(\Gamma_C)$  is needed in the description of   $\partial_a \tME(G)$ and $\partial_a\ME(G)$.
We postpone the lengthy proof   to Section~\ref{sec:proof-prop}.

 \begin{proposition}\label{prop:tool-boundary}
Any graph  $G$    
 satisfies the following properties (PK) and (PC):
\begin{itemize}
\item[(PK)] For each clique $K\in \MK_G$, there exists  $x\in \partial\ME(G)$ such that 
$\det x[K]=0$, 
$\det x[K']\ne 0$ for all $K'\in \MK_G\setminus \{K\}$, and $\Gamma_C(x)\ne 0$ for all $C\in\MC_G$. 
\item[(PC)]  For each cycle $C\in\MC_G$, there exists $x\in\partial \ME(G)$ such that $\Gamma_C(x)=0$,  
$\det x[K]\ne 0$ for all $K\in \MK_G$, and $\Gamma_D(x)\ne 0$ for all $D\in \MC_G\setminus \{C\}$.
\end{itemize}
In addition, any graph $G$ also satisfies the properties (PK) and (PC), where we replace the   condition $x\in \partial \ME(G)$ by the condition $x\in\partial \tME(G)$.
\end{proposition}

We now have all tools in hand to prove Theorem \ref{thm:boundary-cycle-completable}. 
 
\begin{proof}[Proof of Theorem \ref{thm:boundary-cycle-completable}]
It suffices  to show the equality $\partial_a\tME(G)=V(P_G)$ for any graph $G$, since this implies $\partial_a\ME(G)=V(P_G)$ if $G$ is cycle completable as $\ME(G)=\tME(G)$ in that case.

Let $\MX=\partial_a \tME(G)$ denote
 the algebraic boundary of $\tME(G)$ and set 
$\MY=V(P_G)$. We saw earlier in Proposition \ref{proposition:VPn} that $\partial \tME(G)\subseteq V(P_G)$. Hence, we have the inclusion $\MX\subseteq \MY$; we  now show that equality holds.
For this, let  $\MX=\bigcup_{l=1}^L \MX_l$ be the decomposition of the variety $\MX$ into its  irreducible components. 
Recall that the  set $\tME(G)=S(G)\cap \cos(\pi\MET(G))$ is convex semialgebraic with non-empty interior, since it is the intersection of two convex semialgebraic sets (recall Lemma \ref{lem:cosMET}) and the origin lies in its interior.
Hence, by Proposition \ref{lemma:algboundconvex}, each irreducible component $\MX_l$ has codimension 1. By the definition of the polynomial $P_G$ in (\ref{eq:boundaryG}), the variety $\MY$ decomposes as
$$\MY=\bigcup_{ e\in \MK_{G,2}} \Big [ V(1-x_e) \cup V(1+x_e) \big ] \cup \bigcup_{K\in \MK_{G,\ge 3}} V(\det (x[K])) \cup \bigcup_{C\in \MC_G} V(\Gamma_C) \, = \, \bigcup_{j=1}^M \MY_j,
$$
where each component $\MY_j$ is irreducible of codimension 1 (using Lemma \ref{lem:det-irreducible} for irreducibility of $V(\det[K])$), and with  $M=2|\MK_{G,2}|+|\MK_{G,\ge 3}| + |\MC_G|$. Since $\MX\subseteq \MY$, it follows that $\MX_l=\MY_ l$ for $l=1,\ldots,L$, after suitably relabeling the components of $\MY$. 
Thus, $L\le M$. We show that $L=M$, which implies $\MX=\MY$, as desired. Assume for contradiction that $M>L$. Pick an index $k$ such that $L+1\le k\le M$. Now, we use   Proposition \ref{prop:tool-boundary} applied to the set $\tME(G)$: there exists a point $x\in \partial \tME(G) \cap \MY_k$ such that $x\not\in \MY_j$ for all $j\in [M]\setminus \{k\}$.   But, $x\in \partial \tME(G) \subseteq \MX$, which implies $x\in \bigcup_{l=1}^L\MX_l=\bigcup_{l=1}^L \MY_l$, yielding a contradiction.
\end{proof}

\subsection{Applications}\label{sec:applications}
We mention here some  applications and extensions of the above results about algebraic boundaries of graph elliptopes.

\subsubsection*{Which graph elliptopes are spectrahedra?}
As a first application, we investigate when the elliptope of a graph is a spectrahedron.
Clearly, any graph elliptope $\ME(G)$ is a {\em spectrahedral shadow}, i.e., the projection of a spectrahedron. 
However, $\ME(G)$ is in general not a spectrahedron. As we observe below, the elliptope of a cycle $\ME(C_n)$ ($n\ge 4$) is not a spectrahedron, thus rectifying a claim in \cite[Section 4.2]{Sturmfels2010MultivariateGaussians}. As an application of Theorem~\ref{thm:boundary-cycle-completable}, we shall prove that $\ME(G)$ is a spectrahedron precisely when   $G$ is chordal; moreover, we show that this is equivalent to requiring that the algebraic boundary of the elliptope $\ME(G)$ does not intersect its interior
(Corollary \ref{cor:spectrahedron}).

We begin with recalling some well-known facts about spectrahedra (see, e.g., \cite{BPT-2012} for more information). A spectrahedron is a set of the form
\begin{align}\label{eq:Kd}
\MK=\Big\{x\in \R^n: B(x)= B_0+\sum_{i=1}^n x_i B_i\succeq 0\Big\},
\end{align}
where $B_0,B_i\in \MS^N$ for some $N\ge 1$. 
If there exists $x\in \R^n$ such that $B(x)\succ 0$, then $\MK$ is full-dimensional and its interior is
$\text{\rm int}(\MK) =\{x\in\R^n: B(x)\succ 0\}$. 
Hence, its boundary is given by $\partial \MK=\{x\in\MK: \det(B(x))=0\}$, and thus its algebraic boundary satisfies:
$\partial_a \MK \subseteq\{x\in\C^n : \det(B(x))=0\}$.  Therefore, we have
\begin{align}\label{eq:boundary-spec}
\text{If } \MK \text{ is a spectrahedron, then } \text{\rm int}(\MK)  \cap \partial_a \MK=\emptyset.
\end{align}

Note that the inclusion $\partial \MK \subseteq \{x\in\R^n : \det(B(x))=0\}$
 can be strict. 
As an example, the one-dimensional spectrahedron defined by $\left (\begin{smallmatrix}
1 & x & 0 \\ x & 1 & 0 \\ 0 & 0 & x 
\end{smallmatrix}\right )\succeq 0$ is given by the line segment $[0,1]$. 
Its associated determinantal hypersurface is given by the polynomial equation $x(1-x^2)=0$, while the algebraic boundary is given by $x(1-x)=0$. 

By the above discussion, we obtain that $\ME(C_n)$ is not a spectrahedron, and more generally we can characterize the graphs $G$ for which $\ME(G)$ is a spectrahedron.

\begin{corollary}\label{cor:spectrahedron}
For a graph $G=(V,E)$ the following assertions are equivalent:
\begin{itemize}
\item[(i)] $G$ is chordal.
\item[(ii)] The elliptope $\ME(G)$ is a spectrahedron.
\item[(iii)] $\partial_a\ME(G)\cap \text{\rm int} (\ME(G))=\emptyset$.
\end{itemize}
\end{corollary}

\begin{proof}
The implication (i) $\Longrightarrow$ (ii) follows using Theorem \ref{theo:chordal}, and  (ii) $\Longrightarrow$ (iii) follows from the above discussion, see relation (\ref{eq:boundary-spec}).
Now we show (iii) $\Longrightarrow$ (i), using contradiction. Assume that $G$ is not chordal. We show that  $\partial_a\ME(G)\cap \text{\rm int} (\ME(G))\ne \emptyset$.  By assumption, $G$ contains a chordless cycle $C$ of length at least 4. Observe that the elliptope $\ME(C)$ can be obtained as an affine section of the elliptope of $G$. Indeed,  for   $y\in \R^{E(C)}$,  consider the point
$x=(y,\azero)\in \R^E$ obtained by padding $y$ with zero entries at the positions of the edges in $E \setminus E(C)$. Then, $y\in\ME(C)$ if and only if $x\in \ME(G)$. Set $\MH=\{x\in\R^E: x_e=0 \text{ for } e\in E\setminus E(C)\}$. The map $y\mapsto  (y,\azero)$ establishes a one-to-one correspondence between $\ME(C)$ and $\ME(G)\cap \MH$,   between   $\text{\rm int} (\ME(C))$ and $\text{\rm int} (\ME(G)) \cap \MH$, and between   $\partial \ME(C)$ and $ \partial \ME(G)\cap \MH$, which implies
\begin{align*}
 \partial_a\ME(C) \, \simeq \,  \overline{\partial\ME(G)\cap \MH}^z \, \subseteq \, \partial_a \ME(G)\cap \MH.
\end{align*}
By Remark~\ref{rem:strictinclusion},  there exists a point $y\in \text{\rm int}( \ME(C))$ such that $\Gamma_C(y)=0$.
 Thus, $y\in\partial_a\ME(C)$. Hence, the point $x=(y,\azero)$ belongs to $\text{\rm int}(\ME(G))\cap \partial_a\ME(G)$, as desired.
\end{proof}

\subsubsection*{The algebraic boundary of the cone $\mathscr{C}(G)$}

\noindent As a direct application of Theorem \ref{thm:boundary-cycle-completable}, one can compute the algebraic boundary of the cone $\mathscr{C}(G)$ in (\ref{eq:coneCG}) when $G$ is cycle completable. It is provided by the analog of the polynomial $P_G$ in (\ref{eq:boundaryG}), where we now replace the cycle polynomial $\Gamma_C$ by the homogeneous one $\Gamma_C^h$ (introduced in Definition \ref{def:hGamma}). For $\mathscr{C}(C_n)$ the next result recovers Theorem 4.8 in \cite{Sturmfels2010MultivariateGaussians}.

\begin{corollary}\label{cor:coneCG}
If $G$ is cycle completable, then the algebraic boundary of the cone $\mathscr{C}(G)$ is the hypersurface defined by the polynomial equation 
\[
\prod_{K\in\MK_G} \det s[K] \cdot \prod_{C\in\MC_G}\Gamma_C^h(s)=0,
\]
 in the variables $s=(s_{kk}: k\in V; s_{ij}: \{i,j\}\in E)$.
\end{corollary}

\begin{proof}
We use the analog of Proposition \ref{proposition:piboundEn} stating that ${\rm int}(\mathscr{C}(G)) = \pi_{V \cup E}({\rm int}({\cal S}_+^n))$.
Assume that $s\in \R^{V\cup E}$ belongs to $\mathscr{C}(G)$. Then, $s\in \partial \mathscr{C}(G)$ precisely when $s_{kk}=0$ for some $k\in V$, or the point $x\in\R^E$ from (\ref{eq:substitution}) belongs to $\partial \ME(G)$. 
By Theorem \ref{thm:boundary-cycle-completable}, we have that $\partial_a(\ME(G))=V(P_G)=V(\prod_{K\in\MK_G} \det x[K] \cdot\prod_{C\in\MC_G}\Gamma_C(x))$. The result now follows after applying the change of variables and using the definition of $\Gamma_C^h$. We also use the fact that $\mathscr{C}(G)\cap\{s: s_{kk}=0\} \subseteq \mathscr{C}(G)\cap\{s: \det s[K]=0\}$ if $K\in\MK_G$ contains the vertex $k\in V$.
\end{proof}

\subsubsection*{The elliptope of a suspension graph }

As another application, we indicate how to compute the algebraic boundary for the elliptope of a  suspension graph. 
Given a graph $G=(V=[n],E)$, its {\em suspension graph}  $\nabla G$ is  obtained by adding a new node to $G$, called an {\em apex node} and denoted by $0$, which is adjacent to all nodes of $G$. 
So, $V(\nabla G)=V\cup\{0\}$ and $E(\nabla G)= E\cup \{\{0,i\}: i\in V\}$.
As an example,   the wheel graph  $W_k=\nabla C_{k-1}$ on $k$ vertices is obtained by adding an apex node to the cycle $C_{k-1}$. So, $W_4=K_4$ is cycle completable, but  the wheel $W_k$ is not cycle completable for any $k\ge 5$, see Theorem \ref{theo:cycle-completable}.

To compute the algebraic boundary of the elliptope of a suspension graph $\nabla G$ in terms of the algebraic boundary of the elliptope of $G$, we proceed as in \cite[Sec. 4.3]{Sturmfels2010MultivariateGaussians}. We first pass to the cones $\MCscr (\nabla G)$ and $\MCscr (G)$ to allow for arbitrary diagonal entries. We use coordinates $s = (s_{ii}, i \in V(\nabla G);  s_{ij}, (i,j) \in E(\nabla G)) \in \R^{V(\nabla G)\cup E(\nabla G)}$. 
The next observation will be useful.
\begin{lemma} \label{lem:s00}
Let $G = (V = [n], E)$ be any graph and $\nabla G$ its suspension graph. The semialgebraic set $\MCscr(\nabla G) \cap \{ s: s_{00} = 0 \}$ has codimension at least $n + 1$ in $\mathbb{R}^{V(\nabla G) \cup E(\nabla G)}$. 
\end{lemma}
\begin{proof}
If $s \in \MCscr(\nabla G)$, then $s$ admits a positive semidefinite  completion. In particular, the specified $2 \times 2$ minors $s_{00}s_{ii}-s_{0i}^2$ are nonnegative for $i = 1, \ldots, n$. If, in addition to $s \in \MCscr(\nabla G)$, the entry $s_{00}$ is zero, we must have $s_{0i} = 0$ for $i = 1, \ldots, n$. Therefore, $\MCscr(\nabla G) \cap \{ s: s_{00} = 0 \}$ is contained in the coordinate subspace $s_{00}= s_{01} = \cdots = s_{0n} = 0$ of codimension $n + 1$. 
\end{proof}
Since the entries  $s_{0i}$ are specified for all $i\in V = V(G)$, the key idea is to check membership  of $s$ in $\MCscr(\nabla G)$ by taking the Schur complement w.r.t. the entry $s_{00}$.
Restricted to the entries indexed by $V \cup E$, the Schur complement is represented by a map $\varphi: \mathbb{R}^{V(\nabla G) \cup E(\nabla G)} \rightarrow \mathbb{R}^{V \cup E}$:
\begin{align}\label{eq:tphis}
t=\varphi(s)=\Big(s_{ij}- {s_{0i}s_{0j}\over s_{00}}\Big)_{\{i,j\}\in V\cup E}
\end{align}
(setting $\tfrac{s_{0i}s_{0j}}{ s_{00}}=0$ if $s_{00}=0$).

\begin{proposition} \label{prop:suspension}
Let $G$ be any graph. If the algebraic boundary of $\MCscr(G)$ is given by the polynomial $B_G(t) = 0$, then the algebraic boundary of $\MCscr(\nabla G)$ is the zero locus of the numerator of the rational function $B_G(\varphi(s))$. Moreover, each irreducible factor $f(t)$ of $B_G(t)$ pulls back to an irreducible factor of $B_G(\varphi(s))$, namely the numerator of $f(\varphi(s))$. 
\end{proposition}

\begin{proof}
Write $\C^{V(\nabla G)\cup E(\nabla G)}$ as $\C \times \C^V\times \C^{V\cup E}$, where we use the coordinates $s=(s_{00}; s_{0i}, i\in V; s_{ij}, \{i,j\}\in V\cup E)$ on the individual factors. Consider the map
\begin{align*}
\psi: &\  \C^*\times \C^V \times \C^{V\cup E}  \to  \C^*\times \C^V\times \C^{V\cup E} \\
&\  s \mapsto   (s_{00}; s_{0i}, i\in V; \varphi(s)),
\end{align*}
where $\mathbb{C}^* = \mathbb{C} \setminus \{ 0 \}$ and $\varphi(s)$ is as in (\ref{eq:tphis}). The map $\psi$ is an isomorphism of affine varieties. A standard property of the Schur complement implies that, if $s_{00} \neq 0$, then $\varphi(s) \in \MCscr(\nabla G) \Longleftrightarrow s \in \MCscr(G)$. Therefore, the image of $\mathscr{C}(\nabla G) \setminus \{s: s_{00} = 0 \}$ under $\psi$ is
\[ \psi(\mathscr{C}(\nabla G) \setminus \{s: s_{00} = 0 \}) = \R_{>0}\times \R^V\times \mathscr{C}(G). \]
An irreducible component ${\cal Y}_j$ of the algebraic boundary of $\MCscr(G)$ gives an irreducible component $\overline{\psi^{-1}({\cal Y}_j)}^z$ of the algebraic boundary of $\MCscr(\nabla G)$, not contained in $s_{00} = 0$. Here the Zariski closure is taken in $\mathbb{C} \times \mathbb{C}^V \times \mathbb{C}^{V \cup E}$. If ${\cal Y}_j$ is defined by $f(t) = 0$, then $\psi^{-1}({\cal Y}_j)$ is defined by $f(\varphi(s)) = 0$, and its Zariski closure is defined by the numerator of this rational function. The only components of the algebraic boundary of $\MCscr(\nabla G)$ that are not obtained in this way must be contained in $\{ s \, : \, s_{00} = 0 \}$. By Proposition \ref{lemma:algboundconvex} and Lemma \ref{lem:s00}, there is no such~component. 
\end{proof}

\begin{proposition}\label{prop:suspensionboundary}
    Assume  $G=(V,E)$ is cycle completable. 
 The algebraic boundary of the cone $\MCscr (\nabla G)$ is    given by the zero locus of the polynomial
 \[
 \prod_{K\in \MK_G} \det s[\nabla K] \cdot \prod_{C\in\MC_G} \Gamma^h_C(\tilde s)\in \R[s_{00};s_{0i}, i\in V; s_{ij}, \{i,j\}\in V\cup E]
 \]
where we set $\tilde s_{ij}= s_{ij}s_{00}-s_{0i}s_{0j}$ for $\{i,j\}\in V\cup E$ and where 
$\Gamma^h_C$ is the homogeneous cycle polynomial from Definition \ref{def:hGamma}.
Therefore, the     
algebraic boundary of the elliptope   $\ME(\nabla G)$ is given by the zero locus of the polynomial 
    \begin{align}\label{eq:boundaryNablaG}
P^{\nabla}_G(x)= \prod _{K\in\MK_G} \det x[\nabla K] \cdot \prod_{C\in \MC_G} \Gamma^{\nabla}_C(x)\in \R[x_{ij}: \{i,j\}\in E(\nabla G)],
\end{align}
where $\Gamma^{\nabla}_C(x)=\Gamma^h_C(\tilde x)$, with $\tilde x_{ii}=1-x_{0i}^2$ for $i\in V$ and $\tilde x_{ij}=x_{ij}-x_{0i}x_{0j}$ for $\{i,j\}\in   E$. \end{proposition}

\begin{proof}
    We combine Corollary \ref{cor:coneCG} and Proposition \ref{prop:suspension}.
        By Corollary \ref{cor:coneCG},  the algebraic boundary of $\mathscr{C}(G)$ is the vanishing set of the polynomial \[ \prod_{K\in\MK_G} \det t[K] \cdot \prod_{C\in\MC_G}\Gamma_C^h(t)\in \R[t_{ij}: \{i,j\}\in V\cup E]. \] 
        By Proposition \ref{prop:suspension}, we must compute the numerator of $f(\varphi(s))$ for each irreducible factor $f(t)$ of this polynomial. 
    First, consider a factor $f(t)= \det t[K]$ for a clique  $K\in\MK_G$. Then, $\nabla K=K\cup \{0\}$ is a clique of $\nabla G$ and   $\det s[\nabla K]=s_{00}\cdot \det t[K]$ (by taking a Schur complement). So, the component  $\det t[K]=0$ of $\partial_a\MCscr (G)$ contributes the component $\det s[\nabla K]=0$ of~$\partial _a\MCscr(\nabla G)$.
    
    Second, consider a factor $f(t)= \Gamma_C^h(t)$ for a cycle $C\in\MC_G$. 
    Using Theorem \ref{thm:multidegGammah} (and the comment therefater with all $\lambda_i=s_{00}^{1/2}$), we obtain that 
    $\Gamma^h_C(\tilde s)= s_{00}^d\cdot \Gamma_C^h(t)$, where $d=|C|\cdot 2^{|C|-3}$ and $\tilde s_{ij}=s_{00}s_{ij}-s_{0i}s_{0j}$ for $\{i,j\}\in V\cup E$. So, the component  $\Gamma_C^h(t)=0$ of $\partial_a \MCscr (G)$ pulls back to the component  $\Gamma^h_C(\tilde s)=0$ of $\partial_a\MCscr (\nabla G)$. 
\end{proof}

\begin{example}\label{ex:EW4}
The wheel graph  $W_k=\nabla C_{k-1}$ on $k$ vertices is the suspension graph of the cycle $C_{k-1}=(1,2,\ldots, k-1)$, obtained by adding an apex node $0$ to $C_{k-1}$. By Proposition~\ref{prop:suspensionboundary},   the algebraic boundary of its elliptope $\ME(W_{k})$ is given by the polynomial
    \[
   \Gamma^{\nabla}_{C_{k-1}}(x) \cdot    \prod_{\{i,j\}\in E_{k-1}} \det x[\{0,i,j\}].
    \]
   Comparing with  the polynomial $P_{W_{k}}(x)=\Gamma_{C_{k-1}}(x) \cdot \prod _{\{i,j\}\in E_{k-1}}\det x[\{0,i,j\}]$, i.e., the algebraic boundary of $\tilde \ME(W_k)$, it is clear that these two  polynomials   only differ in one  factor. 
In other words, we have $\partial_a\ME(W_k)\ne V(P_{W_k})$ for any $k\ge 5$.
        
    To give a concrete example, consider the wheel $W_5$. The polynomial $\Gamma_4^{\nabla}$ has degree $16$ and depends on all variables $\{x_e:e\in E(W_5)\}$, while $\Gamma_4$ has degree $8$ and only depends on the variables $x_e$ indexed by $e\in E(C_4)$. Figure \ref{fig:EW4} on the left shows the intersections of  $V(\Gamma_4)$ (in yellow) and $V(\Gamma_4^{\nabla})$ (in red) with the three-dimensional affine space defined by $x_{01}=x_{02}=x_{03}=x_{04}=x_{14}=\tfrac{1}{2}$. The first slice is a Lissajous hypersurface, while the second one is the same algebraic variety scaled by a factor smaller than $1$ and tangent to the former.
    The right part of Figure \ref{fig:EW4} compares the 2-dimensional slices of $\ME(W_5)$ (in red) and of $\widetilde\ME(W_5)$ (in yellow), where the slice is now obtained by setting  $x_{01}=x_{02}=x_{03}=x_{04}=x_{14}=\tfrac{1}{2}$ and $x_{12}=x_{23}$.
    \begin{figure}
    \centering
    \includegraphics[height = 5cm]{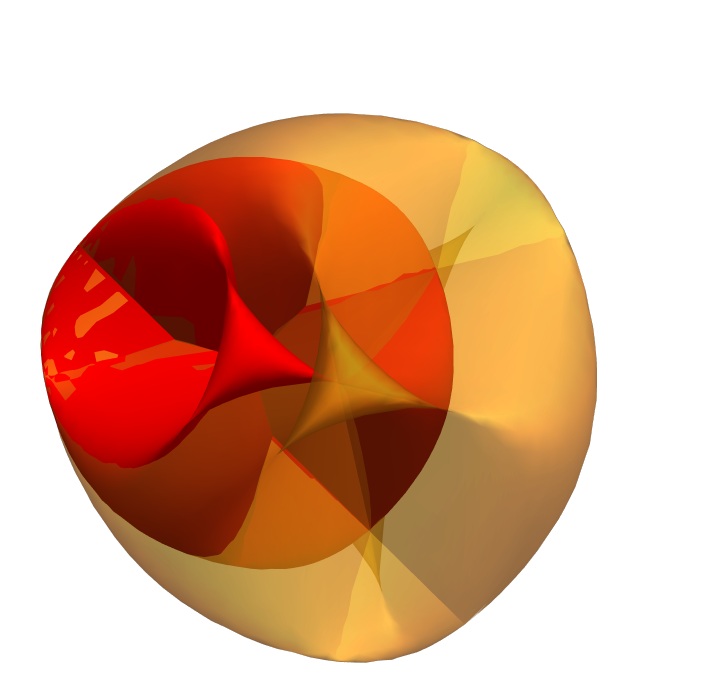} \quad  \quad \quad \quad 
    \includegraphics[height = 5cm]{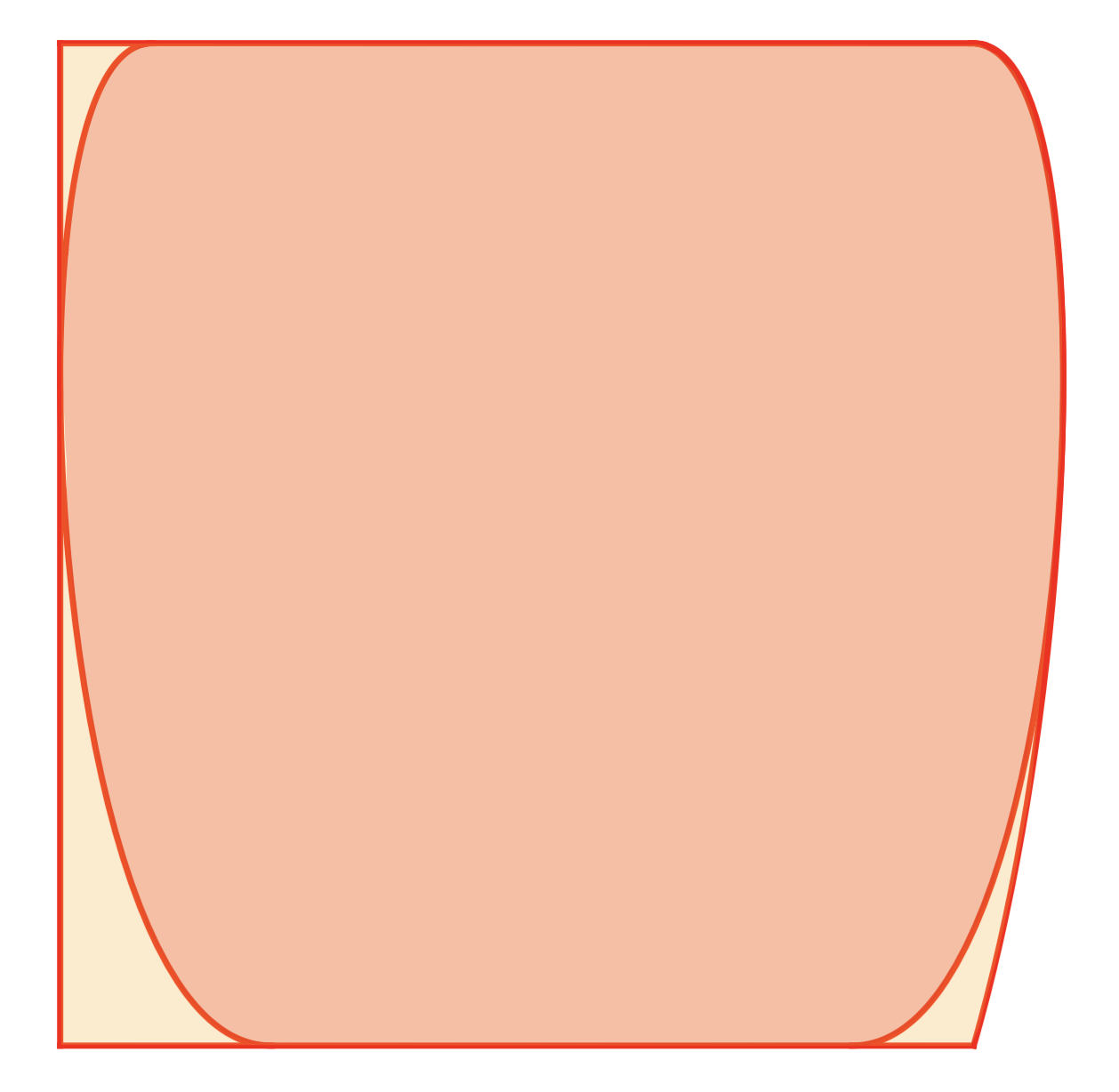}
    \caption{On the left, the 3-D slices of $V(\Gamma_4)$ (in yellow) and  $V(\Gamma_4^{\nabla})$ (in red) at $x_{01}=\dotsc=x_{04}=x_{14}=\tfrac{1}{2}$; on the right, the 2-D slices of $\ME(W_5)$ (in red) and its superset $\widetilde\ME(W_5)$ (in yellow) at $x_{01}=\dotsc=x_{04}=x_{14}=\tfrac{1}{2}$ and $x_{12}=x_{23}$ (discussed in Example \ref{ex:EW4}).}
    \label{fig:EW4}
    \end{figure}
    \end{example}
   
As just observed in Example \ref{ex:EW4}, we have $\partial_a\ME(W_k)\ne V(P_{W_k})$ for any $k\ge 5$. In \cite[Example 4.10]{Sturmfels2010MultivariateGaussians} the authors observe that $\partial_a \ME(H)\ne V(P_H)$ if $H=\widetilde W_4$ is the graph on 5 nodes obtained by splitting one vertex in $W_4=K_4$. We expect that the same holds for all iterated splittings of $W_k$ ($k\ge 5$) and $\widetilde W_4$, and that this can be used to show the converse implication: if $\partial_a\ME(G)=V(P_G)$ then $G$ is cycle completable. We leave this as future work.

\subsection{ Proof of Proposition \ref{prop:tool-boundary}}\label{sec:proof-prop}

Here, we show Proposition \ref{prop:tool-boundary}. Recall that $V(\MF)=\{x\in\C^n: f(x)=0 \text{ for } f\in\MF\}$ is the complex affine variety defined by $\MF\subseteq \R[x_1,\ldots,x_n]$. We write $V_\R(\MF)=V(\MF)\cap\R^n$ for its real points, i.e., the set of all common real roots of $\MF$. We begin with showing the result for the  elliptope $\ME(G)$. After that, we will derive the result for the superset $\tME(G)$ from it.

 Let us fix some notation.
Let $G=(V,E)$ be a graph. 
Let $V=U\cup W$, where $W=V\setminus U$ and $U=K$ for some $K\in \MK_G$, or $U=V(C)$ for some $C\in\MC_G$, depending whether we wish to prove  property (PK) or   (PC). Let $H=G[W]=(W,E(H))$ denote the subgraph of $G$ induced by $W$. 
Our proof strategy is to construct a matrix $X\in \ME_V$ with block-form
 \begin{align}\label{eq:ZplusY}
 X=Z\oplus Y=\begin{pmatrix} Z & \azero\cr \azero & Y\end{pmatrix},
 \end{align}
where $Z\in\ME_U$, $Y\in\ME_W$ are indexed  by $U,W,$ respectively. Then, the point $x=\pi_E(X)$ belongs to $\ME(G)$ and the goal is to select the matrices $Z,Y$ in such a way that $x$ satisfies property (PK) or (PC). 
We treat these two cases separately. 

\smallskip\noindent
\underline{We first show that (PK) holds.} Consider a clique $K \in \MK_{G}$, set $p=|K|$ and $U=K$.
Define the matrix $Z={1\over p-1}(pI_p-J_p)\in \ME_U$, where $J_p$ is the all-ones $p\times p$ matrix. Then, $Z\succeq 0$ with zero eigenvalue for the all-ones eigenvector and eigenvalue $p\over p-1$ with multiplicity $p-1$; hence, each strict principal submatrix of $Z$ is positive definite.
By construction,   $x[K]=Z$, so $\det x[K]=0$.
We now need to select the matrix $Y\in \ME_W$. We will choose $Y$ to be positive definite. This guarantees that $\det x[K']\neq0$ for all cliques $K'\in\MK_G\setminus\{K\}$. Indeed, we have $x[K']=Z[K'\cap K]\oplus Y[K'\cap W]$, and both direct summands are positive definite by construction. 

We must also ensure that $\Gamma_C(x)\ne 0$ for all $C\in\MC_{G}$. We claim that this is satisfied for all choices of $Y$ such that $y = \pi_{E(H)}(Y) \in \mathbb{R}^{E(H)}$ avoids a hypersurface ${\cal H} \subseteq \mathbb{R}^{E(H)}$. Then, a point $y \in {\rm int}({\cal E}(H)) \setminus {\cal H}$ exists, because ${\cal E}(H)$ is full-dimensional in $\mathbb{R}^{E(H)}$ while $\MH$ has codimension one.
To prove our claim, we describe the hypersurface ${\cal H}$ explicitly. For each cycle $C \in {\cal C}_G$, let $f_C \in \mathbb{R}[y_e \, : \, e \in E(C) \cap E(H)]$ be the polynomial obtained by substituting 
\[ x_e \, = \, \begin{cases} y_e & e \in E(C)\cap E(H) \\ 
\frac{-1}{p-1} & e \in E(C)\cap E(K) \\ 
0 & e \in E(C) \setminus (E(K) \cup E(H))
\end{cases}
\]
in the cycle polynomial $\Gamma_C \in \mathbb{Z}[x_e \, : \, e \in E(C)]$. 
By construction, we have $\Gamma_C(x) = 0$ if and only if $f_C(y) = 0$.
Since $C$ is a chordless cycle with at least four edges, we have $|E(C)\cap E(K)|\le 1$ and thus the set $E(C) \cap E(H)$ is non-empty. Hence, by Lemma \ref{lem:specializeGamma}, the polynomial $f_C$ is not identically zero.  We now define the hypersurface ${\cal H}\subseteq \R^{E(C)}$ as 
\[\MH=\Big\{y\in \R^{E(C)}: \prod_{C\in \MC_G} f_C(y)=0\Big\}.\]
In summary, as $\text{\rm int}(\ME(H))\not\subseteq \MH$, one can pick $y \in  {\rm int}({\cal E}(H)) \setminus {\cal H}$ and a positive definite matrix $Y\in\ME_W$ such that $\pi_{E(H)}(Y) = y$. Then, one constructs $X$ as in \eqref{eq:ZplusY} and  the point  $x = \pi_E(X) \in \ME(G)$ satisfies the assertions of property (PK) for the clique $K$. It is clear that $x \in \partial \ME(G)$, since it admits no positive definite completion (because $\det x[K]=0$).

\smallskip\noindent
\underline{Now, we show property (PC).} Consider a cycle $C\in \MC_{G}$ and set $U=V(C)$. Again, we want to  find matrices $Z\in \ME_U$ and $Y\in\ME_W$ so that, for  the block-matrix $X$ in (\ref{eq:ZplusY}),  
the point $x=\pi_E(X)$ satisfies (PC), i.e., $x\in\partial \ME(G)$, 
$\Gamma_C(x) = 0$, $\Gamma_{D}(x) \neq 0$ for  $D \in \MC_G \setminus \{C\}$, and $\det x[K] \neq 0$ for $K \in \MK_G$. We first indicate 
conditions that we need to impose on $Z,Y$.

 By Proposition~\ref{proposition:cosboundMET}(i), 
$\partial \ME(C)=\cos(\pi \partial \MET(C))$ holds. 
Consider an edge $e_0\in E(C)$ and the 
 hyperplane $H_{e_0}=\{a\in\R^{E(C)}: a_{e_0}=\sum_{e\in E(C)\setminus \{e_0\}} a_e\}$. This hyperplane  defines the
 facet $\MF_{e_0}=H_{e_0}\cap \MET(C)$ of $\MET(C)$,
and the relative interior of $\MF_{e_0}$ satisfies 
$\text{\rm relint}(\MF_{e_0})\subseteq (0,1)^{E(C)}\cap \partial \MET(C)$. Define the semialgebraic set
$S_C=\cos( \pi \ \text{\rm relint}(\MF_{e_0}))$. Then, $S_C\subseteq (-1,1)^{E(C)}\cap V_\R(\Gamma_C)$ (recall Remark~\ref{rem:cosLn}).
Moreover, $S_C$ has dimension $|C|-1$: the facet $\text{\rm relint}(\MF_{e_0})$ and its image under the cosine map, which is one-to-one, have the same dimension.
The matrix $Z$   will be chosen in order to ensure that the point $z = \pi_{E(C)}(Z)$ lies in the set $ S_C$, but not  in a real variety ${\cal H}$ of dimension at most $|C|-2$. This variety $\MH$ will be described below.  Its dimension ensures that $S_C \setminus {\cal H} \neq \emptyset$. So, one can pick $z\in S_C \setminus {\cal H}$ and $Z\in \ME_U$ such that $z=\pi_{E(C)}(Z)$, which  guarantees that $\Gamma_C(z)=0$. 
 
Once we have fixed $Z$ such that $z=\pi_{E(C)}(Z) \in S_C \setminus {\cal H}$, we shall pick $Y \in\ME_W$ positive definite such that $y = \pi_{E(H)}(Y) \in {\rm int}(\ME_W) \setminus {\cal H}_z$. Here, ${\cal H}_z$ is a hypersurface in $\mathbb{R}^{E(H)}$ that depends on $z$, to be described explicitly below. 

Note that if $\pi_{E(C)}(Z) \in S_C$ and $Y\succ 0$ in \eqref{eq:ZplusY}, then it is guaranteed that $\det x[K] \neq 0$ for all $K \in \MK_G$. Indeed, we have
$x[K]=Z[K\cap U]\oplus Y[K\cap W]$, with $Y[K\cap W]\succ 0$ since $Y\succ 0$. Moreover, $|K\cap U|\le 2$ as $C$ is chordless with  at least 4 edges. 
So, $Z[K\cap U]$ is either the scalar 1 if $|K\cap U|=1$, or the matrix 
$\left(\begin{smallmatrix} 1& z_{uv}\cr z_{uv}& 1\end{smallmatrix}\right)\succ 0$ if $K\cap U=\{u,v\}$ (since $z_{uv}\ne \pm 1$).

The varieties ${\cal H} \subseteq \mathbb{R}^{E(C)}$   and ${\cal H}_z \subseteq\mathbb{R}^{E(H)}$     are determined by the conditions $\Gamma_D(x) \neq 0$ for all cycles $D \in {\cal C}_G \setminus \{C\}$. We first define ${\cal H}$.

Write ${\cal C}_G^C$ for the set of cycles $D\in {\cal C}_G \setminus \{C\}$ with $E(D)\cap E(H)=\emptyset$.
Let $D \in {\cal C}_G^C$. Consider the polynomial $g_D \in \mathbb{R}[z_e \, : \, e \in E(D) \cap E(C)]$ obtained by specializing
\[ x_e \, = \, \begin{cases} z_e & e \in E(D)\cap E(C) \\ 
0 & e \in E(D) \setminus E(C)
\end{cases}\]
in the cycle polynomial $\Gamma_D \in \mathbb{Z}[x_e \, : \, e \in E(D)]$. By construction, we have that $\Gamma_D(x) = 0$ if and only if $g_D(z) = 0$, where $x = 
\pi_{E(D)}(X)$ with $X$ as in \eqref{eq:ZplusY}. 
Consider the   variety  $\MH_D=V(\Gamma_C) \cap V(g_D)$.  We have strict inclusion  $\MH_D\subsetneq V(\Gamma_C)$, because   
 $g_D(z)$ involves a strict subset of the variables of $\Gamma_C(z)$ (since $E(D)\not\subseteq E(C)$). As $\Gamma_C$ is an irreducible polynomial, its variety $V(\Gamma_C)$ is irreducible and thus the variety $\MH_D$ must have smaller dimension than $V(\Gamma_C)$. Hence, the dimension of $\MH_D$ is at most $|C|-2$, and thus also the dimension of the real variety $\MH_D\cap\R^{E(C)}$ is at most $|C|-2$.
We now define the  real variety $\MH\subseteq \R^{E(C)}$ as  
\[ {\cal H} \, =  \Big\{z\in V_\R(\Gamma_C): \prod_{D \in {\cal C}^C_G} g_D(z) = 0\Big\}
=\bigcup_{D\in\MC^C_G} (\MH_D\cap\R^{E(C)}).\]
By the above reasoning,   $\dim {\cal H}  \leq |C| - 2$, and thus $S_C \setminus {\cal H} \neq \emptyset$ 
 (since $\dim S_C =|C|-1$).

The remaining cycles serve to specify the hypersurface ${\cal H}_z\subseteq \R^{E(H)}$. Let ${\cal C}_G^H \subseteq \MC_G \setminus \{C\}$ be the set of cycles $D\in\MC_G$ with $|E(D)\cap E(H)|\ge 1$.
For each $D \in \MC_G^H$, define the polynomial  $h_{D,z} \in \mathbb{R}[y_e \, : \, e \in E(D) \cap E(H)]$ obtained by specializing 
\[ x_e \, = \, \begin{cases} z_e & e \in E(D)\cap E(C) \\ 
y_e & e \in E(D)\cap E(H) \\
0 & e \in E(D) \setminus (E(C) \cup E(H))\\
\end{cases}
\]
in the cycle polynomial $\Gamma_D \in \mathbb{Z}[x_e \, : \, e \in E(D)]$. Here, $z_e$ is the coordinate of $\pi_{E(C)}(Z)$ labeled by $e\in E(C)$ for our fixed matrix $Z$ from above. By Lemma \ref{lem:specializeGamma}, $h_{D,z}$ is not the zero polynomial. 
The hypersurface ${\cal H}_z $ is defined as 
\[ {\cal H}_z \, = \, \Big\{ y \in  \mathbb{R}^{E(H)} \, : \, \prod_{D \in {\cal C}^H_G} h_{D,z}(y)  = 0 \Big\}. \]
The set ${\rm int}(\ME(H)) \setminus {\cal H}_z$ is non-empty, since $\ME(H)$ is full-dimensional while $\MH_z$ has codimension 1.
 So, pick $y\in {\rm int}(\ME(H)) \setminus {\cal H}_z$, pick $Y\in \ME_W$ positive definite such that $y=\pi_{E(H)}(Y)$, and consider the matrix $X$ from \eqref{eq:ZplusY}. Then, $x=\pi_E(X)\in\ME(G)$ satisfies the conditions from property (PC). 
Remains only to verify that  $x\in\partial \ME(G)$. Indeed, if $x\in \text{\rm int}(\ME(G))$,
then $x=\pi_E(\widehat X)$ for some positive definite $\widehat X\in \ME_n$. But then this would imply
 that $z=\pi_{E(C)}(\widehat X)$ belongs to $\text{\rm int}(\ME(C))$, contradicting the fact that we selected $z\in S_C\subseteq \partial \ME(C)$. 

This concludes the proof of Proposition \ref{prop:tool-boundary} for the case of the elliptope $\ME(G)$, where we ask for a point $x\in \partial \ME(G)$ in properties (PK) and (PC).

 Consider now the case of the superset $\tME(G)$, where we ask for a point $x\in\partial \tME(G)$ in properties (PK) and (PC). We use the fact that 
the boundary of $\tME(G)$ is given by
\[\partial \tME(G) =(\partial S(G) \cap \cos(\pi\MET(G))) \cup (S(G)\cap\cos(\pi \partial \MET(G))).\]
We  show that the point $x\in \partial \ME(G)$ constructed in the above proof, in fact, also belongs to $\partial \tME(G)$. First of all, $x$ clearly belongs to $\tME(G)$. Moreover, for property (PK), $x$ satisfies $\det x[K]=0$ for some   $K\in \MK_G$, implying $x\in\partial S(G)\cap \tME(G)\subseteq \partial \tME(G)$.
For property (PC),   $x=\cos(\pi a) $ with  $a\in \partial \MET(G)$, implying
 $x\in \cos(\pi\partial \MET(G))\cap \tME(G)\subseteq \partial \tME(G)$.
This concludes the proof of Proposition \ref{prop:tool-boundary} for the case of $\tME(G)$.

\subsection*{Acknowledgements}
Our work has been supported by the European Union's HORIZON-MSCA-2023-DN-JD programme under the Horizon Europe (HORIZON) Marie Sklodowska-Curie Actions, grant agreement
101120296 (TENORS).

\bibliographystyle{plain}

\bibliography{bibliography.bib}

@article{BS-2020,
    author={G. Blekherman and K. Shu},
    title={Sums of squares and sparse semidefinite programming},
    journal={SIAM Journal on Applied Algebra and Geometry},
    volume={5},
    number={4},
    pages={651--674},
    year={2020}
}

@inproceedings{Pataki,
    author={G. Pataki},
    title={The geometry of semidefinite programming},
    booktitle={Handbook of Semidefinite Programming},
    editor={H. Wolkowicz and R. Saigal and L. Vandenberghe},
    publisher={Springer},
    series={International Series in Operations Research \& Management Sciences},
    volume={27},
    year={2000}
 }

@article{Belk-Connelly,
    author={M. Belk and R. Connelly},
    title={Realizability of graphs},
    journal={Discrete and Computational Geometry},
    volume={37},
    pages={125--137},
    year={2007}
}

@article{Belk,
    author={M. Belk},
    title={Realizability of graphs in three dimensions},
    journal={Discrete and Computational Geometry},
    volume={37},
    pages={139--162},
    year={2007}
}

@article{LV-MP-2014,
    author={M. Laurent and A. Varvitsiotis},
    title={A new graph parameter related to bounded rank positive semidefinite matrix completions},
    journal={Mathematical Programming},
    volume={145},
    number={1-2},
    pages={291--325},
    year={2014}
}

@article{NLV-JCTB,
    author={M. E.-Nagy and M. Laurent and A. Varvitsiotis},
    title={Forbidden minor characterizations for low-rank optimal solutions to semidefinite programs over the elliptope},
    journal={Journal of Combinatorial Theory Series  B},
    volume={108},
    pages={40--80},
    year={2014}
}

@inproceedings{NLV-Fields,
     author={M. E.-Nagy and M. Laurent and A. Varvitsiotis},
     title={Complexity of the positive semidefinite matrix completion problem with a rank constraint},
     booktitle={Discrete Geometry and Optimization},
     editor={K. Bezdek and  A. Deza and Y. Ye},
     series={Fields Institute Communications},
          publisher={Springer},
     volume={69},
     pages={105--120},
     year={2013}    
}

@article{Solus-Uhler-Yoshida,
    author={L. Solus and C. Uhler and R. Yoshida},
    title={Extremal positive semidefinite matrices whose sparsity pattern is given by graphs without {$K_5$} minors},
    journal={Linear Algebra and its Applications},
    volume={509},
    pages={247--275},
    year={2016}
}

@article{Laurent-Poljak_1995,
    author={M. Laurent and S. Poljak},
    title={On a positive semidefinite relaxation of the cut polytope},
    journal={Linear Algebra and its Applications},
    volume={223-224},
    pages={439--461},
    year={1995}
}

@inproceedings{Laurent_2001,
    author={M. Laurent},
    title={Matrix Completion Problems},
    booktitle={Encyclopedia of Optimization},
    editor={C.A. Floudas and P.M. Pardalos},
    publisher={Springer},
    pages={1311--1319},
    year={2001} 
    }

@article{Laurent_1997,
    author={M. Laurent},
    title={Cuts, matrix completions and graph rigidity},
    journal={Mathematical Programming},
    volume={79},
    pages={255--283},
    year={1997}
}

@article{Laurent_1998,
    author={M. Laurent},
    title={A connection between positive semidefinite and Euclidean distance matrix completion problems},
    journal={Linear Algebra and its Applications},
    volume={273},
    pages={9--22},
    year={1998}
}

@article{GW_1995,
    author={M.X. Goemans and D.P. Williamson},
    title={Improved approximation algorithms for maximum cuts and satisfiability problems using semidefinite programming},
    journal={Journal of the ACM},
    volume={42},
    pages={1115--1145},
    year={1995}
}

@book{BPT-2012,
   author={G. Blekherman and P.A. Parrilo and R.R. Thomas},
   title={Semidefinite Optimization and Convex Algebraic Geometry},
   publisher={American Mathematical Society},
   series={MOS-SIAM Series on Optimization},
   volume={13},
   year={2012}
}

@book{Deza-Laurent-1997,
     author={M.M. Deza and M. Laurent},
     title={Geometry of Cuts and Metrics},
     publisher={Springer},
     series={Algorithms and Combinatorics},
          volume={15},
     year={1997}
}

@article{Johnson-McKee,
    author={C.R. Johnson and T.A. McKee},
    title={Structural conditions for cycle completable graphs},
    journal={Discrete Mathematics},
    volume={159},
    issue={1-3},
    pages={155--160},
    year={1996}
}

@article{Barahona-Mahjoub,
    author={F. Barahona and R. Mahjoub},
    title={On the cut polytope},
    journal={Mathematical Programming},
    volume={36},
    pages={157--173},
    year={1986}
}

@article{GJSW_1984,
    author={R. Grone and C.R. Johnson and E.N. S\'a and H. Wolkowicz},
    title={Positive definite completions of partial hermitian matrices},
    journal={Linear Algebra and its Applications},
    volume={58},
    pages={109--124},
    year={1984}
}

@article{BJL-cycle-completability,
    author={W.W. Barrett and C.R. Johnson and R. Loewy},
    title={The real positive definite completion problem: cycle completability},
    journal={Memoirs of the American Mathematical Society},
    volume={122},
    year={1996}
}

@article{Barrett1993realpositivedefinite,
 author = {W.W. Barrett  and C.R. Johnson and P. Tarazaga},
 title = {The real positive definite completion problem for a simple cycle},
 journal = {Linear Algebra and its Applications},
 volume={192},
 pages={3--31},
 year={1993}
}

@article{Laurent1997realpositivesemidefinite,
 author = {M. Laurent},
 title = {The real positive semidefinite completion problem for series-parallel graphs},
 journal = {Linear Algebra and its Applications},
 volume = {252},
 pages = {347--366},
 year = {1997}
}

@article{Sturmfels2010MultivariateGaussians,
 author = {B. Sturmfels and C. Uhler},
 title = {Multivariate {Gaussians}, semidefinite matrix completion, and convex algebraic geometry},
journal = {Annals of the Institute of Statistical Mathematics},
 volume = {62},
 number = {4},
 pages = {603--638},
 year = {2010},
}

@article{mascarin2025lissajousvarieties,
url = {https://doi.org/10.1515/advgeom-2026-0008},
title = {Lissajous varieties},
author = {Francesco Maria Mascarin and Simon Telen},
pages = {263--282},
volume = {26},
number = {2},
journal = {Advances in Geometry},
doi = {doi:10.1515/advgeom-2026-0008},
year = {2026},
lastchecked = {2026-04-15},
eprint={2509.06844},
archivePrefix={arXiv},
primaryClass={math.AG},
url={https://arxiv.org/abs/2509.06844}
}

@article{sinn2015algebraicboundary,
 author = {R. Sinn},
 title = {Algebraic boundaries of convex semi-algebraic sets},
 journal = {Research in the Mathematical Sciences},
  volume = {2},
  number={3},
  year = {2015}
}

@book{cox2025ideals,
 author = {D.A. Cox and J. Little and D. O'Shea},
 title = {Ideals, varieties, and algorithms. {An} introduction to computational algebraic geometry and commutative algebra (to appear)},
 edition = {5th edition},
 isbn = {978-3-031-91840-7; 978-3-031-91843-8; 978-3-031-91841-4},
 year = {2025},
 publisher = {Cham: Springer},
 language = {English},
 zbMATH = {8074064}
}

@article{BelAfiaMeroniTelen2025Chebyshev,
  author    = {Z. Bel-Afia and C. Meroni and S. Telen},
  title     = {Chebyshev varieties},
  journal   = {Mathematics of Computation},
  year      = {2025}
}

@book{CoxLittleOShea2005UsingAG,
  title     = {Using Algebraic Geometry},
  author    = {D.A. Cox and J. Little and D. O'Shea},
  publisher = {Springer},
  series    = {Graduate Texts in Mathematics},
  volume    = {185},
  year      = {2005},
  address   = {New York}
}

\end{document}